\documentclass[sn-mathphys]{sn-jnl2}

\usepackage{lipsum}
\usepackage{amsfonts}
\usepackage{graphicx}
\usepackage{epstopdf}
\usepackage{geometry}
\usepackage{graphbox}
\usepackage{caption}
\usepackage{subcaption}
\usepackage{amsmath,amssymb,amsthm}
\usepackage{amssymb}
\usepackage{bm}
\usepackage{commath,mathtools}
\usepackage[capitalize]{cleveref} 

\jyear{2023}%

\theoremstyle{thmstyleone}%
\newtheorem{theorem}{Theorem}
\theoremstyle{thmstyletwo}%
\newtheorem{remark}{Remark}%
\newtheorem{lemma}{Lemma}%
\newtheorem{corollary}{Corollary}%

\theoremstyle{thmstylethree}%

\renewcommand{\v}[1]{\bm{#1}}

\newcommand{\pars}[1]{\left( #1 \right)}%
\newcommand{\braces}[1]{\left\{ #1 \right\}}%

\newcommand{\conj}[1]{\overline{#1}}

\newcommand{\reals}{\mathbb{R}}
\newcommand{\CC}{\mathbb{C}}
\newcommand{\ZZ}{\mathbb Z}
\newcommand{\zplus}{{\mathbb Z}^+}
\newcommand{\Surf}{S}
\newcommand{\opint}{\mathcal{I}}
\newcommand{\opintTh}{\opint[\Theta_p](\v x)}
\newcommand{\opquadsym}{\operatorname{Q}}
\newcommand{\opremsym}{\operatorname{E}}

\newcommand{\oprem}{\opremsym_{n}}
\newcommand{\opremTh}{\oprem [\Theta_p](\v x)}

\newcommand{\opquadt}{\opquadsym_{t,n_t}}
\newcommand{\opremt}{\opremsym_{t,n_t}}
\newcommand{\opquadphi}{\opquadsym_{\varphi,n_\varphi}}
\newcommand{\opremphi}{\opremsym_{\varphi,n_\varphi}}

\newcommand{\opremtphi}{\opremsym^2_{n_t,n_{\varphi}}}

\renewcommand{\Im}{\operatorname{Im}}

\newcommand{\est}{\operatorname{est}}

\newcommand{\expcoord}{\v C}
\newcommand{\expcoordlm}{\expcoord_{\ell}^m}
\newcommand{\atan}{\operatorname{atan2}}
\newcommand{\gammaaxisym}{\v\gamma^{\circ,A}}
\newcommand{\geomfac}[1]{G_{\v\gamma, #1}}
\newcommand{\Eest}{E_{\v\gamma}^{EST}}
\newcommand{\EQQ}{E_{\v\gamma}^{Q}}
\newcommand{\EQwarg}{\EQQ(f,p,n_t,n_\varphi,\v x)}
\newcommand{\Eestwarg}{\Eest(f,p,n_t,n_\varphi,\v x)}
\newcommand{\wtz}{w^{TZ}}
\newcommand{\wgl}{w^{GL}}

\newcounter{estno}
\newtheorem{estimate}[estno]{Error estimate}
\crefname{estimate}{Error estimate}{Error estimates}

\begin{document}

\title[C. Sorgentone and A.-K. Tornberg]{Estimation of quadrature errors for layer potentials
    evaluated near surfaces with spherical topology}


\author*[1]{\fnm{Chiara} \sur{Sorgentone}}\email{chiara.sorgentone@uniroma1.it}

\author[2]{\fnm{Anna-Karin} \sur{Tornberg}}\email{akto@kth.se}

\affil*[1]{\orgdiv{Department of Basic and Applied Sciences for Engineering}, \orgname{Sapienza University of Rome}, \orgaddress{\country{Italy}}}

\affil[2]{\orgdiv{Department of Mathematics}, \orgname{KTH Royal Institute of Technology}, \orgaddress{\city{Stockholm}, \country{Sweden}}}


\abstract{Numerical simulations with rigid particles, drops or vesicles
constitute some examples that involve 3D objects with spherical
topology. When the numerical method is based on boundary integral
equations, the error in using a regular quadrature rule to
approximate the layer potentials that appear in the formulation will
increase rapidly as the evaluation point approaches the surface and
the integrand becomes sharply peaked. To determine when the accuracy
becomes insufficient, and a more costly special quadrature method
should be used, error estimates are needed. 
In this paper we present quadrature error estimates for layer
potentials evaluated near surfaces of genus 0, parametrized using a
polar and an azimuthal angle, discretized by a combination of the
Gauss-Legendre and the trapezoidal quadrature rules. The error estimates involve
no unknown coefficients, but complex valued roots of a specified
distance function. The evaluation of the error estimates in general requires a
one dimensional local root-finding procedure, but for specific
geometries we obtain analytical results. Based on these explicit
solutions, we derive simplified error estimates for layer potentials
evaluated near spheres; these simple formulas depend only on the
distance from the surface, the radius of the sphere and the number of
discretization points. The usefulness of these error estimates is
illustrated with numerical examples.}




\maketitle

\section{Introduction}

We consider a generic layer potential over a regular surface $S \subset \reals^3$,
\begin{align}
  u(\v x) = \int_{\Surf} \frac{k \pars{\v x, \v y}\sigma(\v y)
  }{\norm{\v y - \v x}^{2p}} \dif S(\v y) ,
\label{eq:layerpot_S2}
\end{align}
where $2p\in \zplus$ and the evaluation (or target) point $\v x \in \reals^3$ is
allowed to be close to, but not on,  $S$.
The functions $k \pars{\v x, \v y}$ and $\sigma(\v y)$ as
well as $S$ are assumed to be smooth.
When $\v x$ is close to $S$, the integrand will be peaked around the
point on $S$ closest to $\v x$, implying that, while the integral is well defined
analytically, it is difficult to resolve well numerically.

In paper \cite{AFKLINTEBERG20221}, we derived estimates for the
numerical errors that result when applying quadrature rules to such
layer potentials. Specifically, we considered the panel based
Gauss-Legendre quadrature rule and the global trapezoidal rule. The
estimates that was derived have no unknown coefficients and can be
efficiently evaluated given the discretization of the surface. The
evaluation involves a local one-dimensional root finding procedure. In
numerical experiments, we have found the estimates to be both
sufficiently precise and computationally cheap to be practically
useful. This means that they can be used to determine when the regular
quadrature is insufficient for a required accuracy, and hence when a
more costly special quadrature method must be invoked.  In deriving
these estimates, we assumed that the local (for Gauss-Legendre) or
global (for trapezoidal rule) surface parametrization is such that
the map between the parameter space and the surface coordinates is
one-to-one.


For surfaces of genus $0$, topologically equivalent to a sphere, it is
quite common to use a global parametrization in two angles,
i.e. spherical coordinates.  At the {\em poles} of such a
parametrization, the parameter to surface coordinate map is not
one-to-one. Here, the derivative of the surface coordinate with
respect to the azimuth angle vanish, and with it, the surface area
element. For such parametrizations, the previously derived error
estimates cannot be directly used for all evaluation points $\v x$. 

It is quite common in applications of boundary integral equations to
discretize surfaces with parametrizations based on spherical
coordinates, and to that attach a Gaussian grid, i.e. a discretization with a Gauss-Legendre
quadrature rule in the polar angle (or a mapping of the polar angle)
and the trapezoidal rule in the periodic azimuthal angle. This has been used for simulations of Stokes flow with solid particles in e.g. \cite{AfKlinteberg2016qbx,Corona2017504}, and with drops and vesicles \cite{Sorgentone2018167,Sorgentone2022,Rahimian2015766,Veerapaneni2016278}.
In the latter
case, the deformable shapes are represented by spherical harmonics
expansions. 

In this paper, we consider the case of a general smooth surface of
genus $0$,
parametrized using a polar and an azimuthal angle, discretized by a
combination of the Gauss-Legendre quadrature rule and the trapezoidal
rule as described above. Before we introduce the contributions of this
paper, we will briefly describe the derivations of the estimates
in paper \cite{AFKLINTEBERG20221}, and the previous results that enabled that work. 


Consider the following simple integrals
\begin{equation}
\int_{E} \frac{1}{\left( (t-a)^2+b^2 \right)^p} \, dt = \int_{E}
\frac{1}{\left(t-z_0 \right)^p \left(t-\bar{z}_0 \right)^p} \, dt,
\label{eq:cartesian_int_basic}
\end{equation}
with $z_0=a+ib$, $a,b \in
\reals$, $b > 0$, $p\in \zplus$, and
\begin{equation}
\int_{E} \frac{1}{\left(t-z_0 \right)^p} \, dt, \quad z_0=a+ib, \quad a,b \in
\reals, b \ne 0, p\in \zplus,
\label{eq:complex_int}
\end{equation}
over a basic interval $E$, e.g. $[-1,1]$ for Gauss-Legendre and
$[0,2\pi)$ for the trapezoidal rule.
If $b$ is small the integrands have two poles/one pole close to the
integration interval along the real axis.
The theory of Donaldson and Elliott \cite{Donaldson1972} 
defines the quadrature error as a contour integral in the complex
plane over the integrand multiplied with a so-called remainder
function, that depends on the quadrature rule.
Elliott et al. \cite{Elliott2008} derived error estimates for the error in the approximation of
(\ref{eq:cartesian_int_basic}) with an $n$-point
Gauss-Legendre quadrature rule. To estimate the contour
integral, they used residue calculus for $p=1$ and branch cuts for
$0<p<1$. 
In \cite{AfKlinteberg2016quad},  af Klinteberg and Tornberg derived
error estimates for both (\ref{eq:cartesian_int_basic}) and (\ref{eq:complex_int}) for the
Gauss-Legendre quadrature rule, for any $p \in \zplus$. Corresponding
results were derived also for the trapezoidal rule, but for
integration over the unit circle. Previous studies on the trapezoidal rule include the survey by Trefethen and Weideman \cite{Trefethen2014}, and the error bound provided by Barnett in \cite{Barnett2014} for the quadrature error in evaluating the harmonic double layer potential.

In \cite{AfKlinteberg2018}, a key step was taken to accurately
estimate the quadrature errors for the Gauss-Legendre quadrature rule
in the approximation of layer potentials in 2D, written in complex
form. Introducing a parametrization of a smooth curve segment,
$\gamma(t) \in \CC$, a typical form of an integral to evaluate is
\begin{equation}
\int_E \frac{f(t) \, \gamma'(t)}{(\gamma(t)-z_0)^p} \, dt.
\label{eq:complex_int_curved}
\end{equation}
As compared to the estimates for the simple complex integral
(\ref{eq:complex_int}) above, the estimates derived for this integral
require the knowledge of $t_0 \in \CC$ such that $\gamma(t_0)=z_0$. In
practice, given the Gauss-Legendre points used to discretize the
panel, a numerical procedure is used to compute $t_0$. 
Note that not all layer potentials in 2D can be written in this form.
Using the same techniques, remarkably accurate error estimates were
derived also for layer potentials for the Helmholtz and Stokes
equations, in \cite{AfKlinteberg2018} and \cite{Palsson2019},
respectively.

Now, let $S$ in  (\ref{eq:layerpot_S2}) be a curve $\Gamma \in
\reals^2$,  $\Gamma=\gamma(E)$, $E  \subset \reals$. 
We can then write the layer potential in (\ref{eq:layerpot_S2}) in the equivalent form
\begin{align}
  u(\v x) = \int_E \frac{k \pars{\v x, \v \gamma(t)}\sigma(\v\gamma(t)) }{\norm{\v\gamma(t)-\v x}^{2p}} \norm{\v\gamma'(t)} \dif t
  = \int_E \frac{f(t) \dif t}{\norm{\v\gamma(t)-\v x}^{2p}}.
\label{eq:base_integral}
\end{align}
In the last step all the components that are
assumed to be smooth have been collected in the function $f(t)$, which has an implicit
dependence on $\v x$.
The location of the poles are in this case given by each $t_0 \in \CC$
such that the denominator is zero. In \cite{AFKLINTEBERG20221}, error estimates were
derived for such integrals, both for the Gauss-Legendre and the
trapezoidal quadrature rules. To evaluate these estimates, the
pair of the complex conjugate roots $\{t_0,\bar{t_0}\}$ closest to integration interval is needed, and
is in practice found through a numerical root-finding procedure. The
estimates will be stated in Section \ref{sec:basictheory}. 

A key observation in \cite{AFKLINTEBERG20221} was that the error estimates for the
numerical approximation of (\ref{eq:base_integral}) can be derived the
same way for  $\Gamma \in \reals^d$, $d=2,3$. The only difference is
that $R^2(t,x)=\norm{\v\gamma(t)-\v x}^{2}$ will have three additive
terms instead of two. Starting with the curve estimates in $\reals^3$,
error estimates for the prototype layer potential
(\ref{eq:layerpot_S2}) were derived. 

With $S$ a two-dimensional surface in $\reals^3$, parametrized by
  $\v \gamma: E \rightarrow \reals^3$, $E=\left\{ E_1 \times E_2 \right\}
  \subset \reals^2$, (\ref{eq:layerpot_S2}) takes the form
\begin{align}
  u(\v x) = \iint_E 
  \frac{k \pars{\v x, \v \gamma(t,\varphi)}\sigma(\v\gamma(t,\varphi)) }
  {\norm{\v\gamma(t,\varphi)-\v x}^{2p}} 
  \norm{\dpd{\v\gamma}{t} \times \dpd{\v\gamma}{\varphi}}
  \dif \varphi \dif t
  = \iint_E \frac{f(t,\varphi) \dif \varphi \dif t}{\norm{\v\gamma(t,\varphi)-\v x}^{2p}}.
\label{eq:3Dpot}
\end{align}
All the components that are assumed to be
smooth have here been collected in $f(t,\varphi)$ which depends implicitly on $\v x$. 
In \cite{AFKLINTEBERG20221} error estimates were derived for the numerical approximation
of (\ref{eq:3Dpot}) by composite Gauss-Legendre quadrature or global trapezoidal
quadrature. Numerical examples were shown with tensor product
quadrature rules based on Gauss-Legendre quadrature for surface
discretizations of quadrilateral patches, and the tensor product
trapezoidal rule for global discretizations. 

\section{Contributions and outline}
\label{sec:section2}
In this paper, we will discuss the generalization of the results of
paper \cite{AFKLINTEBERG20221} to the case of smooth surfaces topologically equivalent to a
sphere, 
parametrized by 
  $\v \gamma^\circ: U \rightarrow \reals^3$, where 
$U=\left\{ (\theta,\varphi)\in [0,\pi]\times [0,2\pi) \right\}$.

A generic such surface can be represented by
\begin{align}
\gamma^\circ(\theta,\varphi)=\sum_{\ell=0}^{\infty} \sum_{m=-\ell}^{\ell}
  \expcoordlm Y_{\ell}^m(\theta,\varphi),
\label{eq:gamma_sph_harm}
\end{align}
where  $\expcoordlm \in \CC^3$ and $Y_{\ell}^m(\theta,\varphi)$ is the spherical harmonic function of
degree $\ell$ and order $m$.

The surface $\gamma(t,\varphi)$ in (\ref{eq:3Dpot}) can then be defined as 
\begin{align}
  \gamma(t,\varphi)=\gamma^{\circ}(\theta(t),\varphi),
  \label{eq:gammast}
 \end{align}
 where $t \in [-1,1]$ and $\varphi \in [0,2\pi)$. With this
 parametrization, we can naturally discretize the integral using a $n_{t}$-point
Gauss-Legendre quadrature rule in the $t$ coordinate, and a $n_{\varphi}$
point trapezoidal rule in the periodic $\varphi$ coordinate. 
We will consider two different maps $\theta(t):[-1,1]\rightarrow [0,\pi]$: a simple linear
scaling 
\begin{equation}
\label{eq:linear_map}
\theta(t)=(t+1)\pi/2,
\end{equation} 
and a non-linear one
\begin{equation}
\label{eq:cos_map}
\theta(t)=\cos^{-1}(-t)=\pi-\cos^{-1}(t).
\end{equation}
The inverse mappings are $t(\theta)=-1+2\theta/\pi$ and
$t(\theta)=-\cos(\theta)$, respectively. With both mappings, $t=-1$
corresponds to $\theta=0$.


In section \ref{sec:basictheory}, we introduce the error
estimates derived in \cite{AFKLINTEBERG20221} for the integral over a
curve in $\reals^2$ or $\reals^3$ (\ref{eq:base_integral}). In section
\ref{sec:surface_estimates} we then introduce the extension to a
surface of genus 0 in $\reals^3$ for our specific discretizations,
based on what was done in paper \cite{AFKLINTEBERG20221}.
In section \ref{sec:analytical} we derive analytical results for
axisymmetric surfaces. We also show how, based on the explicit
knowledge of the roots, it is possible to derive simplified error
estimates for layer potentials evaluated near spheres. In section
\ref{sec:numeval}, we discuss how to numerically evaluate the
estimates for a general surface with spherical topology, and in
section \ref{sec:numer-exper} we show how the error estimates perform
on different numerical examples. 

%

\section{Quadrature error estimates for curves in $\reals^2$ and $\reals^3$}
\label{sec:basictheory}
Let us introduce the base interval $E$, which for the Gauss-Legendre
quadrature will be $[-1, 1]$ and for the trapezoidal rule $[0, 2\pi)$.
Consider an integral over such a base interval
\begin{equation}
\opint[g]=\int_{E} g(t) \, dt, 
\end{equation}
and an $n$-point quadrature rule with quadrature nodes $\{t_{\ell}\}_{\ell=1}^n$ and
corresponding quadrature  weights $\{w_{\ell}\}_{\ell=1}^n$ to approximate it, 
\begin{equation}
Q_n[g]=\sum_{\ell=1}^n g(t_{\ell}) w_{\ell}.
\end{equation}
The error 
\begin{equation}
E_n[g]=\opint[g]-Q_n[g],
\end{equation}
as a function of $n$ will depend on the function $g$ and the specific
quadrature rule. 
We will consider closed curves for the trapezoidal rule and open
curves (segments) with the Gauss-Legendre quadrature rule.

We now introduce the squared distance function for a curve in
$\reals^d$ ($d=2$ or $3$), to an evaluation point
$\v x \in \reals^d$,
\begin{align}
  R^2(t,\v x) \coloneqq \norm{\v\gamma(t)-\v x}^{2} 
  = \sum_{i=1}^d (\gamma_i(t)-x_i)^2.
  \label{eq:R2_def}
\end{align}
We will later evaluate this function also for $t\in \CC$, in which case we
will use the right most expression. This expression can
then evaluate as a complex number, and will no longer be a norm. 
Our integral of interest (\ref{eq:base_integral}) can be written in the form
\begin{align}
\opintTh = \int_E \Theta_p(t,\v x) \dif t, \qquad
\Theta_p(t,\v x)=\frac{f(t)}{\pars{R^2(t,\v x)}^p},
\label{eq:cartesian_int}
\end{align}
and we want to estimate  $\opremTh$.

As $\v x$ is not on $\v\gamma(t)$, we have $R^2(t,\v x)>0$ for
$t \in E$.
There will however be complex conjugate pairs of roots to $R^2(t,\v
x)$, since $R^2(t,\v x)$ is real for real $t$.
Let $\braces{t_0, \conj t_0}$ be the pair closest to
$E$, s.t.
\begin{align}
  R^2(t_0,\v x) = R^2(\conj t_0, \v x) = 0.
\end{align}
Under the assumption that $f$ is smooth, the region of analyticity of
$\Theta_p(t,\v x)$ is bounded by these roots. We will henceforth refer to
them both as roots (of $R^2$) and
singularities (of the integrand).
They are in most applications not
known a priori, but can be found numerically for a given target point
$\v x$ (see section \ref{sec:rootfinding}). 

The quadrature error $\opremTh$ can, following Donaldson and Elliott \cite{Donaldson1972}, 
be written as a contour integral in the complex
plane over the integrand $\Theta_p(t,\v x)$ multiplied with a so-called remainder
function, that depends on the quadrature rule.
If $p$ is an integer, $t_0$ and $\bar{t_0}$ are $p$th order
poles of the integrand.  If $p$ is a half-integer, the integrand has
branch points at these singularities. 

An important step in the derivation leading up to estimates of
the quadrature error in \cite{AFKLINTEBERG20221}, is to divide and multiply the integrand
with the factor $t-w$, where $w\in \CC$ is the singularity at the branch
begin considered.  This yields the singularity $(t-w)^{-1}$ to consider in the
complex plane, and introduces what is denoted the geometry factor. 

The \emph{geometry factor} $G$ is, for an evaluation point $\v x \in
\reals^d$ and $w \in \CC$ a root of $R^2$, 
defined as
\begin{align}
  G(w,\v x) = 
  \lim_{t \to w} \frac{t-w}{R^2(t,\v x)}
  =
  \left(2\pars{\v\gamma(w)-\v x}\cdot\v\gamma'(w) \right)^{-1}
  . \label{eq:geometry_factor}
\end{align}
With these definitions, we are ready to state the error estimates from
\cite{AFKLINTEBERG20221}, for both the trapezoidal rule and the Gauss-Legendre quadrature
rule. 

\begin{estimate}[Trapezoidal rule]
  \label{est:trapz_p}
  Consider the integral in (\ref{eq:cartesian_int})
for an evaluation point $\v x \in \reals^d$, 
  with $2p \in \zplus$, where $\gamma(E)$ is the parametrization
  of a smooth closed curve in $\mathbb R^d$
where $d=2$ or $3$.
The integrand is assumed to be periodic in $t$ over the integration interval  $E=[0,2\pi)$.
The error in approximating the integral with the $n$-point trapezoidal
rule can in the limit $n \rightarrow \infty$ be estimated as 
 \begin{align}
  \abs{\opremTh} 
 \approx  
    \frac{4\pi n^{p-1} }{\Gamma(p)}  \abs{f(t_0)} \abs{G(t_0,\v x)}^p e^{-n \abs{\Im{t_0}}}.
    \label{eq:trapz_p}
  \end{align}
Here, $\Gamma(p)$ the gamma function, and the geometry
factor $G$ is defined in  (\ref{eq:geometry_factor}). 
The squared distance function is defined in (\ref{eq:R2_def}), and 
$\braces{t_0, \conj t_0}$ is the pair of complex conjugate roots of this
$ R^2(t,\v x)$ closest to the integration interval $E$. 
\end{estimate}

\begin{estimate}[Gauss-Legendre rule]
  \label{est:gl_p}
 Consider the integral in (\ref{eq:cartesian_int})
for an evaluation point $\v x \in \reals^d$, 
  with $2p \in \zplus$, where $\gamma(E)$ is the parametrization
  of a smooth closed curve in $\mathbb R^d$
 $d=2$ or $3$, with $E=[-1,1]$.
  The error in approximating the integral with the $n$-point
  Gauss-Legendre rule can in the limit $n \rightarrow \infty$ be
  estimated as
\begin{align}
    \abs{\opremTh} \approx     
    \frac{4\pi}{ \Gamma(p)}   
    (2n+1)^{p-1} \abs{f(t_0)}  \abs{G(t_0,\v x)}^p     
    \frac{   \abs{\sqrt{t_0^2-1}}^{1-p}       }{\abs{t_0 + \sqrt{t_0^2-1}}^{2n+1}},
    \label{eq:gl_p}
\end{align}
where $\sqrt{z^2-1}$ is defined as $\sqrt{z+1}\sqrt{z-1}$ with $-\pi <
\arg(z \pm 1) \le \pi$.
  Here, $\Gamma(p)$ the gamma function, the geometry
factor $G$ is defined in  (\ref{eq:geometry_factor}), 
the squared distance function is defined in (\ref{eq:R2_def}), and 
$\braces{t_0, \conj t_0}$ is the pair of complex conjugate roots of this
$ R^2(t,\v x)$ closest to the integration interval $E$. 
\end{estimate}

In paper \cite{AFKLINTEBERG20221}, each estimate was written as two different estimates, one for
positive integers $p$, and one for positive half-integers $p$. The
derivation of the two estimates follows a different path, using residue
calculus and branch cuts, respectively. However, using
$(p-1)!=\Gamma(p)$ for integer $p$ the two estimates can both be written in the form
given above. 

\section{Quadrature errors near two-dimensional surfaces in $\reals^3$}
\label{sec:surface_estimates}

Let us now consider the three-dimensional case, and the prototype
layer potential (\ref{eq:3Dpot}). Here, $S \subset \mathbb R^3$ is a two-dimensional
surface parametrized by $\v\gamma : E \to \mathbb R^3$,
$E = \{E_1 \times E_2\} \subset \mathbb R^2$. As introduced in
(\ref{eq:gammast}), we will specifically consider $\v\gamma(t,\varphi)$ 
where  $t \in E_1=[-1,1]$ and $\varphi \in E_2=[0,2\pi)$. 

In analogy to the squared distance function to a curve, as introduced
in (\ref{eq:R2_def}), we now introduce the squared distance function 
between the surface $\v\gamma(t,\varphi)$ and the evaluation point $\v
x=(x,y,z)$, 
\begin{align}
\begin{split}
  R^2(t,\varphi,\v x) &\coloneqq \norm{\v\gamma(t,\varphi)-\v x}^{2} \\ &=
(\gamma_1(t,\varphi)-x)^2+(\gamma_2(t,\varphi)-y)^2+(\gamma_3(t,\varphi)-z)^2.
  \label{eq:R2_def_surface} 
  \end{split}
\end{align}
Note that we will later evaluate $R^2$ also for complex arguments using the
right most expression, in which case it is no longer a norm.
With this, the integrand of (\ref{eq:3Dpot}) can be written 
\begin{align}
\Theta_p(t,\varphi,\v x)=\frac{f(t,\varphi)}{\pars{R^2(t,\varphi,\v x)}^p}.
  \label{eq:Thetap_st} 
\end{align}

The operators $\opint[g]$, $Q_n[g]$ and $E_n[g]$ were introduced in the beginning of Section
\ref{sec:basictheory}, with $E_n[g]= \opint[g]-Q_n[g]$. Here we use them with a subindex indicating if
they are applied in the $t$ or the $\varphi$ direction, where it
should be understood that Gauss-Legendre quadrature rule is applied in
the $t$-direction, and the trapezoidal rule in the $\varphi$
direction. For ease of
notation, we will skip the brackets above, such that $\opint_t \opint_\varphi \Theta_p$ means an
integration of $\Theta_p$ first in the $\varphi$ and then in the $t$
direction. We can then write the tensor product quadrature as
\begin{align}
  \opquadt \opquadphi \Theta_p
  &= \pars{\opint_t - \opremt}\pars{\opint_{\varphi} - \opremphi}
    \Theta_p, 
\end{align}
and from here
\begin{align}
  \opremtphi \Theta_p &\coloneqq 
             \left(  \opint_t\opint_{\varphi} - \opquadt \opquadphi \right) \, \Theta_p \notag
\\ 
& = \left( \opint_t \opremphi + \opremt  \opint_{\varphi} - \opremt
                        \opremphi  \right) \, \Theta_p  \notag \\
& \approx
             \left( \opint_t \opremphi + \opint_{\varphi}\opremt \right) \, \Theta_p
             .
             \label{eq:proderr_approx}
\end{align}
In this last step, we have neglected the quadratic error term, and used that
$\opremt\opint_{\varphi}=\opint_{\varphi}\opremt$. This last fact
follows from
$\opremt\opint_{\varphi}=(\opint_{t}-\opquadt)\opint_{\varphi}$
combined with $\opint_{t}\opint_{\varphi}=\opint_{\varphi}\opint_{t}$
and $\opquadt \opint_{\varphi}=\opint_{\varphi}\opquadt$. For a more
detailed discussion, see \cite{Elliott2011}.
For some basic integrals, Elliott et al. \cite{Elliott2015} have shown 
that the quadratic error term that we here neglect can
have an important contribution. As it is a higher order
contribution, this is only true when the quadrature error is
large, and we will derive our estimates without it, as was also done in \cite{AFKLINTEBERG20221}. 

Explicitly writing out the first term in the right hand side of
(\ref{eq:proderr_approx}), we have 
\begin{align}
  \opint_t \opremphi \Theta_p
  =
  \int_{E_1}
  \left[
  \int_{E_2}
  \frac{f(t,\varphi) \dif \varphi}{\norm{\v\gamma(t,\varphi)-\v x}^{2p}}
  -
  \sum_{l=1}^{n_\varphi} 
  \frac{f(t,\varphi_l) \wtz_l}{\norm{\v\gamma(t,\varphi_l)-\v x}^{2p}}
  \right] \dif t,
  \label{eq:IsEt}
\end{align}
where $\varphi_l= 2\pi (l-1) /n_\varphi$, $l=1,\ldots,n_\varphi$ and
$\wtz_l=2\pi/n_\varphi$, $\forall l$. 
The term in the brackets (i.e. $\opremphi$) represents the quadrature
error of the trapezoidal rule on the line $L_t$ that for a given $t$ is defined as
\begin{align}
  L_t \coloneqq
  \left\{
  \v\gamma(t,\varphi)
  \mid
  \varphi \in E_2
  \right\}
  ,\quad
  t \in E_1
  .
\end{align}
For short, we will denote this curve $\v\gamma(t,\cdot)$.
For a fixed $t$, this is the quadrature error for the trapezoidal
rule, for which an
estimate is given in \cref{est:trapz_p}. 
The term $\opint_\varphi \opremt \Theta_p$ can be written analogously
to   \eqref{eq:IsEt}, simply swapping $t$ and $\varphi$, introducing
the Gauss-Legendre quadrature nodes and weights. 
The error estimate that needs to be integrated in
this term is given in \cref{est:gl_p}.

To be able to distinguish if the geometry factor in the error estimate
corresponds to $\gamma(t,\cdot)$ or $\gamma(\cdot,\varphi)$, we 
extend the definition of the  geometry factor in
(\ref{eq:geometry_factor}) and denote 
\begin{align}
  \geomfac{1}(t,\varphi) &= \left(2\pars{\v\gamma(t,\varphi)-\v x}\cdot\v\gamma_t(t,\varphi) \right)^{-1}
  =\left( \frac{\partial}{\partial t} R^2(t,\varphi,\v x)\right)^{-1},
 \label{eq:geometry_factor_one}                  
  \\
   \geomfac{2}(t,\varphi)&= \left(2\pars{\v\gamma(t,\varphi)-\v x}\cdot\v\gamma_\varphi(t,\varphi) \right)^{-1}
  =\left( \frac{\partial}{\partial \varphi} R^2(t,\varphi,\v x)\right)^{-1}.
\label{eq:geometry_factor_two}
\end{align}

We need to work with the absolute value of the error, and we
will use the estimate
\begin{align}
\abs{\opremtphi \Theta_p}
             \approx
             \abs{\left( \opint_t \opremphi + \opint_{\varphi}\opremt \right) \,
  \Theta_p}
  \le \opint_t \abs{ \opremphi   \Theta_p } +  \opint_{\varphi}
 \abs{ \opremt   \Theta_p }.
\label{eq:proderr_approx_abs}
\end{align}
Expanding back from this shorthand notation, this can be formulated as follows.
\begin{estimate}[Surface in $\reals^3$]
   \label{est:3D}
 Given an evaluation point $\v x \in \reals^3$,  consider the integral in  (\ref{eq:3Dpot}) with $2p \in \zplus$,
   where  $S \subset \mathbb R^3$ is a two-dimensional
smooth closed surface parametrized by $\v\gamma : E \to \mathbb R^3$,
$E = \{E_1 \times E_2\} = [-1,1] \times [0,2\pi) \subset \mathbb R^2$.
The integrand is assumed to be periodic in $\varphi$ over the
integration interval  $E_2=[0,2\pi)$.

The error in approximating the integral with the $n_t$ point
Gauss-Legendre rule in the $t$-direction and the $n_\varphi$-point trapezoidal
rule in the $\varphi$ direction is defined as
\begin{align}
 \EQwarg = \abs{
  \int_{E_1}  \int_{E_2}
  \frac{f(t,\varphi) }{\norm{\v\gamma(t,\varphi)-\v x}^{2p}} \dif \varphi  \dif t
  -  \sum_{l=1}^{n_\varphi} \sum_{k=1}^{n_t} 
  \frac{f(t_k,\varphi_l) \wtz_l \wgl_k}{\norm{\v\gamma(t_k,\varphi_l)-\v x}^{2p}}}
  \label{eq:EQdef}
\end{align}
where $\{t_k,\wgl_k\}_{k=1}^{n_t}$ and $\{\varphi_l,\wtz_l\}_{l=1}^{n_\varphi}$ are the
Gauss-Legendre and trapezoidal rule quadrature nodes and weights. 

Assume that $\varphi_0(t,\v x)$ and $t_0(\varphi, \v x)$ as defined
below exist for $t \in [-1,1]$ and $\varphi \in [0, 2 \pi)$,
respectively. Then,  $\EQQ$
can be estimated as 
 \begin{align}
  \EQwarg  \approx \Eestwarg
             =
  E_{\v \gamma}^{TZ}(f,p,n_\varphi,\v x)+ E_{\v \gamma}^{GL}(f,p,n_t,\v x),
    \label{eq:ETZplusEGL}
 \end{align}
 where
\begin{align}
  E_{\v \gamma}^{TZ}(f,p,n_\varphi,\v x)&=\int_{E_1} 
  \abs{f\pars{t,\varphi_0(t,\v x)} \geomfac{2}\pars{t,\varphi_0(t,\v x),\v x }^p}
  \est^{TZ}(\varphi_0(t, \v x),n_{\varphi},p) \dif t ,
 \label{eq:ETZdef}\\
  E_{\v \gamma}^{GL}(f,p,n_t,\v x) &=  \int_{E_2} 
  \abs{f\pars{t_0(\varphi, \v x),\varphi} \geomfac{1}\pars{t_0(\varphi, \v x),\varphi ,\v x}^p}
  \est^{GL}(t_0(\varphi,\v x),n_t,p) \dif \varphi ,
    \label{eq:EGLdef}
\end{align}
and 
\begin{align}
 \est^{TZ}(\varphi_0,n,p) & =
                      \frac{4\pi}{ \Gamma(p)}  n^{p-1}  e^{-n \abs{\Im \varphi_0}},
  \label{eq:estTZ}
  \\
 \est^{GL}(t_0,n,p) & =
  \frac{4\pi}{ \Gamma(p)}   \abs{ \frac{2n+1}{\sqrt{t_0^2-1}} }^{p-1}
\abs{t_0 + \sqrt{t_0^2-1}}^{-(2n+1)},
\label{eq:estGL}
\end{align}
where $\sqrt{z^2-1}$ is defined as $\sqrt{z+1}\sqrt{z-1}$ with $-\pi <
\arg(z \pm 1) \le \pi$.


Here, $\Gamma(p)$ is the gamma function, and the geometry
factors $\geomfac{1}$ and $\geomfac{2}$  are defined in  (\ref{eq:geometry_factor_one})- (\ref{eq:geometry_factor_two}). 
Given the evaluation point $\v x \in \reals^3$ and $t$,
$\{\varphi_0(t,\v x), \conj{\varphi_0(t, \v x)}\}$ is the pair of
complex conjugate roots of $ R^2(t,\varphi,\v x) =
\norm{\v\gamma(t,\varphi)-\v x}^{2} $ closest to the integration
interval $E_2$, and similarly for $t_0(\varphi,\v x)$ for given $\v x$ and $\varphi$.

\end{estimate}
\begin{remark}
\label{rem:n_inf}
  For this error estimate to be useful, it should give a good approximation of the error already
  for moderate values of $n_t$ and $n_\varphi$.  In practice we find
  this to be true as long as  $n_t$ and $n_\varphi$ are large enough for the surface to be well
  resolved. This will be discussed in the numerical results
  section.
\end{remark}
\begin{remark}
  Given $\v x \in \reals^3$, it is not guaranteed that a root
  $\varphi_0(t,\v x)$ exists for all $t \in [-1,1]$, nor that 
  $t_0(\varphi, \v x)$ exists for all $\varphi \in [0, 2 \pi)$.
 For example,  for surfaces of spherical topology with a global parametrization as
  defined in \cref{eq:gammast}, the squared distance function $ R^2(t,\varphi,\v x)$ is
 independent of $\varphi$ for $t=-1,1$, and hence no root $\varphi_0$
 exists.
 For axisymmetric surfaces, no root $\varphi_0$ exist for evaluation
 points along the axis of symmetry, interior or exterior to the
 surface. 
\end{remark}

The exposition in this section has followed what was done in
\cite{AFKLINTEBERG20221},  however combining
integration by the Gauss-Legendre rule in one direction, and the
trapezoidal rule in the other. In \cite{AFKLINTEBERG20221} discretizations of surfaces of
genus 1 with either a global trapezoidal rule in both directions, or a
panel based Gauss-Legendre rule were considered. When a panel based
discretization is used, the error estimates for the panels closest to
the evaluation point are added together.
We cannot theoretically guarantee that the roots that we need for
evaluation of the error estimate always exist, and there are in
general no analytical formulas for the roots. 
In \cite{AFKLINTEBERG20221}, the root finding is done numerically, and approximations to
  the integrals in \eqref{eq:ETZdef} and \eqref{eq:EGLdef} are made, using the fact that
  the error contribution is strongly localized to the region on the
  surface closest to the evaluation point. 
For each evaluation point $\v x$, only one root
  $t_0(\varphi^*,\v x)$ and one root $\varphi_0(t^*,\v x)$ are needed,
  where $\gamma(t^*,\varphi^*)$ is the grid point (quadrature node) on
  the surface closest to the evaluation point $\v x$.  This approach
  works well  apart from the rare occasions where the 
  root finding algorithm fails for evaluation
  points quite far from the surface.


To understand how we can evaluate error estimates for surfaces of
genus 0 with a global parametrization, we will first analytically consider the simpler
case of an axisymmetric surface, at times further simplified to a
sphere.  We will then in Section \ref{sec:numeval} discuss the practical evaluation of the estimate,
including how to approximate the remaining integrals and determine the
roots as needed.

\section{Analytical derivations for axisymmetric and spherical surfaces}
\label{sec:analytical}
In this section we will consider an
axisymmetric surface, as parametrized by
\begin{align}
\gammaaxisym(\theta,\varphi)= \left( a(\theta) \sin(\theta) \cos(\varphi),
  a(\theta) \sin(\theta) \sin(\varphi), b(\theta) \cos(\theta) \right), 
\label{eq:gamma_axisymm}
\end{align}
with $a(\theta),b(\theta)>0$.
Here, $\gammaaxisym: U \rightarrow \reals^3$, where 
$U=\left\{ (\theta,\varphi)\in [0,\pi]\times [0,2\pi) \right\}$.
For some results, we will simplify further and set $a(\theta)=b(\theta)=a$
and consider a sphere of radius $a$. 

The parametrization $\v \gamma(t,\varphi)$ relates to this parametrization
through a mapping $\theta=\theta(t)$ as given in
\cref{eq:gammast}, with $\v \gamma^\circ=\gammaaxisym$.
The map $\theta(t)$ will change the
parametrization of the surface in $t$ and hence yield different
locations of the Gauss-Legendre quadrature nodes on the surface.
In this section, we will keep this choice open to the extent possible, and state most results
in $\theta$ and $\varphi$.

Note that the axis of symmetry for $\gammaaxisym(\theta,\varphi)$ is
the $z$-axis. For a surface of a different shift and orientation, the
evaluation point $\v x$ can be translated and rotated into a local
coordinate system of the particle, and the results stated below will
apply.

Generally, the roots of $R^2$ cannot be found analytically, and we
need to compute them using a root finding procedure. For the axisymmetric
case, we can however analytically find the roots $\varphi_0$ given
$\theta$, and for a sphere, we can furthermore find the
roots $\theta_0$ given $\varphi$.

We will see that the estimate for the error incurred by the trapezoidal
rule cannot be evaluated for an evaluation point at the symmetry axis,
as the integrand in \cref{eq:ETZdef} becomes undefined. Using the
analytical expressions for the roots, we can study appropriate limits
as the evaluation point approaches the symmetry axis. We will also use
these analytical results to derive a simplified error estimate for the
sphere.

\subsection{Analytical roots to the squared distance function}
  
We will start by finding the roots to the squared distance function defined with respect to a
circle in the $xy$ plane, and then extend this result.

\begin{lemma}[Root of $R^2$ for circle in plane]
Let a circle of radius $a>0$ in the $xy$-plane be parametrized by $(\gamma_1(\alpha),
\gamma_2(\alpha), \gamma_3(\alpha))=a(\cos\alpha,\sin\alpha,0)$, $0 \le \alpha <2 \pi$.
Given a point $\v x =(x,y,z)\in \reals^3$, not on the curve,
define
\begin{align*}
  R^2(\alpha,\v x)=(\gamma_1(\alpha)-x)^2+(\gamma_2(\alpha)-y)^2+z^2.
\end{align*}
Then 
$R^2(\alpha,\v x)=0$ for $\alpha=\alpha_0$ with
\begin{equation}
\alpha_0=\atan(y,x)\pm i \ln
  \left(\lambda+\sqrt{\lambda^2-1}\right), \quad
  \lambda=\frac{1}{2a}\frac{a^2+x^2+y^2+z^2}{\sqrt{x^2+y^2}}.
  \label{eqn:alpha_root_circle}
\end{equation}

{\flushleft Here, $\lambda>1$ and $\atan(\eta,\xi)$ is the argument of the complex number $\xi+i\eta$, 
$-\pi < \atan(\eta,\xi) \le \pi$.  }
\label{lemma:root_circle}

\end{lemma}
\begin{remark}
Note that if $\alpha_0$ is a root to $R^2(\alpha,\v x)$, so is
$\alpha_0+2\pi p$ for any $p \in \ZZ$. Further, notice that we have
\begin{align*}
\alpha_0=\atan(y,x)\pm i \ln
  \left(\lambda+\sqrt{\lambda^2-1}\right) = \atan(y,x) \mp i \ln
  \left(\lambda-\sqrt{\lambda^2-1}\right).
\end{align*}

\end{remark}
\begin{proof} (\cref{lemma:root_circle})
Introduce $\rho^2=x^2+y^2$ and $d^2=a^2+\rho^2+z^2$. With this
notation we have
\begin{align}
R^2(\alpha,\textbf{x})& =(a\cos(\alpha)-x)^2+(a\sin(\alpha)-y)^2+
                        z^2 =d^2-2a(x\cos(\alpha)+y\sin(\alpha))
  \label{eq:R2_phi_circle} \\
& =d^2-2a\rho\left(\cos(\bar{\theta}\right)\cos(\alpha)+\sin(\bar{\theta})\sin(\alpha))
                                \notag
\end{align}
where $x=\rho \cos(\bar{\theta})$,  $y=\rho \sin(\bar{\theta})$ with
$\bar{\theta}=\atan(y,x)$ have been introduced in the last step. 

To determine the roots of $R^2$ we replace $\alpha=\bar{\theta}+i \eta$,
and rewrite the four trigonometric terms similarly to $\cos(\bar{\theta}+i\eta)=\frac{1}{2}
(e^{i\bar{\theta}}e^{-\eta}+e^{-i\bar{\theta}}e^{\eta})$. After
simplification, this yields 
\begin{equation}
0=d^2-2a\rho \left[ \frac{1}{2}(e^{-\eta}+e^\eta) \right]=d^2-a\rho(\beta^{-1}+\beta),
\label{eq:R2_phi_circle3}
\end{equation}
where we have replaced $\eta=\ln(\beta)$.
Introducing $\lambda=\frac{1}{2}\frac{d^2}{a\rho}$, we are left to
solve  $\beta^2-2\lambda \beta+1=0$ which yields $\beta=\lambda \pm \sqrt{\lambda^2-1}$. 

We have  $\alpha_0=\bar{\theta}+i \ln \beta$, for the two values of
$\beta$. Using that 
$\lambda-\sqrt{\lambda^2-1}=(\lambda+\sqrt{\lambda^2-1})^{-1}$ and
hence $\ln (\lambda-\sqrt{\lambda^2-1})=-\ln
(\lambda+\sqrt{\lambda^2-1})$, we can write
\begin{equation}
\alpha_0=\atan(y,x)\pm i\ln(\lambda+\sqrt{\lambda^2-1}) = \atan(y,x) \mp i \ln
  \left(\lambda-\sqrt{\lambda^2-1}\right).
\end{equation}
To see that $\lambda>1$, introduce  $\rho=\sqrt{x^2+y^2}$, and write
$\lambda=(a/\rho+\rho/a)/2+z^2/(2a\rho)$. Here, $(a/\rho+\rho/a)/2 \ge1$
with equality only when $\rho=a$. Since $\v x$ is not on the curve, we
cannot have $z=0$ in this case, and hence $\lambda>1$.
\end{proof}
\begin{lemma}[Root of $R^2$ in $\varphi$ given $\theta$]
Let the axisymmetric surface $\gammaaxisym(\theta,\varphi)$ be
parametrized as in \cref{eq:gamma_axisymm}. Given an evaluation point
$\v x =(x,y,z)\in \reals^3$ not on $\gammaaxisym$,
define
\begin{align*}
  R^2(\theta,\varphi,\v
  x)=(\gammaaxisym_1(\theta,\varphi)-x)^2+(\gammaaxisym_2(\theta,\varphi)-y)^2+(\gammaaxisym_3(\theta,\varphi)
  -z)^2.
\end{align*}
Assume $x^2+y^2>0$.
Given $\theta=\bar{\theta} \in (0,\pi)$, 
$R^2(\bar{\theta},\varphi,\v x)=0$ for $\varphi=\varphi_0$ with
\begin{align}
\varphi_0=\atan(y,x)\pm i \ln
  \left(\lambda+\sqrt{\lambda^2-1}\right), \quad
  \lambda=\frac{1}{2\tilde{a}}\frac{\tilde{a}^2+x^2+y^2+(\tilde{b}-z)^2}{\sqrt{x^2+y^2}}, 
\label{eq:phiroot_given_theta}
\end{align}
where $\tilde{a}=a(\bar{\theta})\sin \bar{\theta}$ and  $\tilde{b}=b(\bar{\theta})\cos \bar{\theta}$.
Here, $\lambda>1$, and $\atan(\eta,\xi)$ is the argument of the complex number
$\xi+i\eta$, $-\pi < \atan(\eta,\xi) \le \pi$.  
\label{lemma:phiroot_given_theta}
\end{lemma}
\begin{remark}
Note that the root (in \cref{lemma:phiroot_given_theta}) is not defined if $x^2+y^2=0$. In this case $R^2(\bar{\theta},\varphi,\v
x)=\tilde{a}^2+(\tilde{b}-z)^2$, which is independent of $\varphi$ and
always positive, so no root can be found. Similarly, if  $\bar{\theta}=0$ or
$\pi$, $R^2(\bar{\theta},\varphi,\v x)$ is again independent of
$\varphi$, and no root can be found.
\label{rem:phiroot_given_theta} 
\end{remark}
\begin{proof} (\cref{lemma:phiroot_given_theta})
  Fix $\theta=\bar{\theta}$, define  $\tilde{a}=a(\bar{\theta})\sin(\bar{\theta})$, $ \tilde{b}=b(\bar{\theta)}\cos(\bar{\theta})$, and rewrite \cref{eq:gamma_axisymm} as 
\begin{equation}
\label{eq:sphere_thetafixed}
\gammaaxisym(\theta,\varphi)=\left( \tilde{a} \cos(\varphi), \tilde{a} \sin(\varphi),\tilde{b} \right).
\end{equation}
Now
\begin{equation}
R^2(\bar{\theta},\varphi,\v x)=(\tilde{a}\cos(\varphi)-x)^2+(\tilde{a}\sin(\varphi)-y)^2+ (\tilde{b}-z)^2.
\end{equation}
\cref{eq:sphere_thetafixed} describes a circle of radius $\tilde{a}$
in the $xy$ plane at $z=\tilde{b}$, and so we can proceed similarly as we did for \cref{lemma:root_circle}. 
In this case, we get
$\lambda=\frac{\tilde{a}^2+(\tilde{b}-z)^2+\rho^2}{2\tilde{a}\rho}$,
where  $\rho^2=x^2+y^2$. With this $\lambda$,  the expression of the
analytical root for $\varphi_0(\bar{\theta})$ is of the same form as
in \cref{eqn:alpha_root_circle}, as given in \cref{eq:phiroot_given_theta}.

To see that $\lambda>1$, the argument is similar to that in the proof
of \cref{lemma:root_circle}. Rewrite $\lambda$ as 
$\lambda=(\tilde{a}/\rho+\rho/\tilde{a})/2+(\tilde{b}-z)^2/(2\tilde{a}
\rho)$. Here, $(\tilde{a}/\rho+\rho/\tilde{a})/2 \ge1$
with equality when $\rho=\tilde{a}$. Since $\v x$ is not on the surface, we
cannot have $z=\tilde{b}$ in this case, and hence $\lambda>1$.
\end{proof}
\begin{lemma}[Root of $R^2$ in $\theta$ given $\varphi$]
Let the sphere $\gammaaxisym(\theta,\varphi)$ be
parametrized as in \cref{eq:gamma_axisymm} with $a(\theta)=b(\theta)=a$.
Given an evaluation point $\v x =(x,y,z)\in \reals^3$ not on $\gammaaxisym$, 
define
\begin{align*}
  R^2(\theta,\varphi,\v
  x)=(\gammaaxisym_1(\theta,\varphi)-x)^2+(\gammaaxisym_2(\theta,\varphi)-y)^2+(\gammaaxisym_3(\theta,\varphi)
  -z)^2.
\end{align*}
Given $\varphi=\bar{\varphi} \in [0,2\pi)$, assume that if $z=0$, then
$\bar{\varphi}-\atan(y,x) \ne \pi/2+p \pi$, $p \in {\mathbb Z}$.
Then 
$R^2(\theta,\bar{\varphi},\v x)=0$ for $\theta=\theta_0$ with
\begin{align}
  \theta_0 &=
    \atan(x \cos \bar{\varphi} + y \sin \bar{\varphi},z)
  \pm i \ln
  \left(\lambda+\sqrt{\lambda^2-1}\right),
  \label{eqn:thetaroot_given_phi}
 \\
  \lambda&=\frac{1}{2a}\frac{a^2+x^2+y^2+z^2}{\sqrt{(x \cos \bar{\varphi} + y \sin \bar{\varphi})^2+z^2}}.
\label{eqn:thetaroot_given_phi2}
\end{align}
Here, $\lambda>1$ and $\atan(\eta,\xi)$ is the argument of the complex number
$\xi+i\eta$,  $-\pi < \atan(\eta,\xi) \le \pi$. 

\label{lemma:thetaroot_given_phi}
\end{lemma}
\begin{remark}
Note that the root in \cref{lemma:thetaroot_given_phi} is not defined for the case when $z=0$ and
$\bar{\varphi}-\atan(y,x)=\pi/2+p \pi$, $p \in {\mathbb
  Z}$.
In this case, $R^2(\theta,\bar{\varphi},\v x)=a^2+x^2+y^2$, which is
independent of $\theta$ and always positive, so no root can be found. 
\end{remark}
\begin{proof} (\cref{lemma:thetaroot_given_phi})
We fix the angle $\varphi=\bar{\varphi} \in [0, 2\pi)$ and want to
determine the root $\theta_0$ of:
\begin{align}
R^2(\theta,\bar{\varphi},\v x)&=(a \sin(\theta) \cos(\bar{\varphi})-x)^2+(a \sin(\theta) \sin(\bar{\varphi})-y)^2+(a \cos(\theta) -z)^2 
\end{align}
After a rotation of the evaluation point by $\bar{\varphi}$ in
clockwise direction around the $z$-axis, we get\\
$(\tilde{x}, \tilde{y}, \tilde{z})=(x\cos(\bar{\varphi})+y\sin(\bar{\varphi}), -x\sin(\bar{\varphi})+y\cos(\bar{\varphi}), z)$.
In this new coordinate system \cref{eq:gamma_axisymm} becomes
$\tilde{\gamma}^\circ(\theta,\bar{\varphi})=\left(a \sin(\theta), 0 ,a
  \cos(\theta)\right)$.
This is as in \cref{lemma:root_circle}, but here with $z$ instead of $x$,
$x$ instead of $y$, $y$ instead of $z$. Also, $\theta\in [0, \pi]$, so
we are considering half circle. Now introduce
$\tilde{\rho}=\sqrt{\tilde{z}^2+\tilde{x}^2}$.
Using \cref{lemma:root_circle}, we get
\begin{equation}
\lambda=\frac{a^2+\tilde{\rho}^2+\tilde{y}^2}{2a\tilde{\rho}}=\frac{a^2+\tilde{x}^2+\tilde{y}^2+\tilde{z}^2}{2a\tilde{\rho}},
\end{equation}
and
\begin{equation}
\theta_0=\atan(\tilde{x},\tilde{z}) \pm i \ln(\lambda+\sqrt{\lambda^2-1}).
\end{equation}
Using the relations between  $(\tilde{x},\tilde{y},\tilde{z})$ and
$x,y,z$,
yields \cref{eqn:thetaroot_given_phi,eqn:thetaroot_given_phi2}.

To see that $\lambda>1$, rewrite $\lambda$ as 
$\lambda=(a/\tilde{\rho}+\tilde{\rho}/a)/2+\tilde{y}^2/(2a
\tilde{\rho})$.
Here, $(a/\tilde{\rho}+\tilde{\rho}/a)/2 \ge1$
with equality when $\tilde{\rho}=a$. Since $\v x$ is not on the surface, we
cannot have $\tilde{y}=0$ in this case, and hence $\lambda>1$.
\end{proof}
\begin{corollary} 
Under the assumptions of \cref{lemma:thetaroot_given_phi}, with the
evaluation point at the $z$-axis
inside or outside of the sphere, $\v x
=(0,0,z)\in \reals^3$, $\abs{z}\ne a$, it holds
\begin{align*}
  \theta_0=
  \pm i \ln(\abs{z}/a),
\end{align*}
independent of $\bar{\varphi}$.
\label{corr:thetaroot_given_phi_on_zaxis}
\end{corollary}
\begin{proof} (\cref{corr:thetaroot_given_phi_on_zaxis})
From \cref{eqn:thetaroot_given_phi,eqn:thetaroot_given_phi2}
we have $\sqrt{\lambda^2-1}=\abs{a^2-z^2}/(2a\abs{z})$, and
$\lambda+\sqrt{\lambda^2-1}$ evaluates as $\abs{z}/a$ for $\abs{z}>a$
and $a/\abs{z}$ for $\abs{z}<a$. The result follows since
$\ln(\abs{z}/a)=-\ln(a/\abs{z})$.
\end{proof}

\subsection{Error estimates for evaluation points close to and on the symmetry axis}
\label{sec:estsymmax}
With $\v \gamma(t,\varphi)=\gammaaxisym(\theta(t),\varphi)$, 
$E^{TZ}(\v\gamma,f,p,n_\varphi,\v x)$ from \cref{est:3D} can be
written as
\begin{align}
  E_{\v \gamma}^{TZ}(f,p,n_\varphi,\v x)&=\frac{4\pi}{ \Gamma(p)}
                                       n^{p-1} \int_{E_1} 
  \abs{f\pars{t,\varphi_0(\theta(t),\v x)} } E_{fac}^{TZ}(\v x,\theta(t))  \dif t ,
    \label{eq:ETZ2}
\end{align}
where
\begin{align}
E_{fac}^{TZ}(\v x,\theta) =   \abs{\geomfac{2}^\circ\pars{\theta,\varphi_0(\theta),\v x }}^p
  e^{-n_\varphi \abs{\Im \varphi_0(\theta,\v x)}}.
\label{eq:EfacTZ_thph}
\end{align}
Here $\varphi_0(\theta,\v x)$ is the root associated with
$\gammaaxisym (\theta,\varphi)$. Similarly,
$\geomfac{2}^\circ\pars{\theta,\varphi,\v x}$ is the second geometry factor
associated with $\gammaaxisym(\theta,\varphi)$, as will be
explicitly defined below.  Note that the
differentiation in \cref{eq:geometry_factor_two} is with respect to
$\varphi$ and hence the mapping of the $t$-coordinate yields no extra
factor. 

As was commented on in \cref{rem:phiroot_given_theta}, the root
$\varphi_0(\theta,\v x)$ is not defined for $\v x=(0,0,z)$, $z \ne
a$. Furthermore, the geometry factor in
$E_{fac}^{TZ}((0,0,z),\theta)$, $z \ne a$ is infinite. To proceed, we
will derive the expression for the general
$E_{fac}^{TZ}(\v x,\theta)$, where $\v x$ is not on the $z$-axis, and
then consider the appropriate limit to understand the behavior for
evaluation points along the $z$-axis.

\begin{theorem}
  Let the axisymmetric surface $\gammaaxisym(\theta,\varphi)$ be
  parametrized as in \cref{eq:gamma_axisymm}.
Given an evaluation point $\v x =(x,y,z)\in \reals^3$, not on $\gammaaxisym$, 
let $R^2(\theta,\varphi,\v x)$ and $\varphi_0(\theta,\v x )$, the root
of $R^2$, be defined as in \cref{lemma:phiroot_given_theta}, and define
$\geomfac{2}^\circ\pars{\theta,\varphi,\v x}=\left(
  \frac{\partial}{\partial \varphi} R^2(\theta,\varphi,\v
  x)\right)^{-1}$.

Let $E_{fac}^{TZ}(\v x,\theta)$ be as in \cref{eq:EfacTZ_thph} with 
$2p \in \zplus$, $n_{\varphi} \in \zplus$, and 
assume that $\rho^2=x^2+y^2>0$. Then for $\theta \in (0,\pi)$ it holds
\begin{align}
  E_{fac}^{TZ}(\v
  x,\theta) & =\abs{\geomfac{2}^\circ\pars{\theta,\varphi_0(\theta, \v x),\v x }}^p
              e^{-n_\varphi \abs{\Im \varphi_0(\theta,\v x)}}
\notag \\
              & =\frac{1}{(\tilde{a}^2+\rho^2+(\tilde{b}-z)^2)^p}
  \left( \frac{\lambda}{\sqrt{\lambda^2-1}} \right)^p
  \left( \frac{1}{\lambda+\sqrt{\lambda^2-1}}  \right)^{n_{\varphi}},
\label{eq:EfacTZ}
\end{align}
where $\tilde{a}=a(\theta)\sin(\theta)$, $\tilde{b}=b(\theta)\cos(\theta)$  and $\lambda$ is as defined in \cref{eq:phiroot_given_theta} in 
\cref{lemma:phiroot_given_theta}. We have that $\lambda>1$ since $\v x
=(x,y,z)\in \reals^3$ is not on $\gammaaxisym$.

For $\lambda>\lambda_0>1$, it holds
\begin{align}
  E_{fac}^{TZ}(\v x, \theta) \le \frac{C}{(\tilde{a}^2+\rho^2+(\tilde{b}-z)^2)^p}
  \left( \frac{1}{2\lambda}  \right)^{n_{\varphi}},
  \label{eq:EfacBound}
\end{align}
where
$C= (\lambda_0/(\lambda_0-1))^p (2\lambda_0/(2\lambda_0-1))^ {n_{\varphi}}$.
For large $\lambda_0$, $C \approx 1$. 

\label{thm:EfacTZ_est}
\end{theorem}
\begin{proof} (\cref{thm:EfacTZ_est})
Introduce the angle $\nu=\atan(y,x)$ in the $xy$-plane, such that
$x=\rho \cos\nu$, $y=\rho \sin \nu$. With this, we can write
$\partial R^2(\theta,\varphi,\v x)/\partial
  \varphi=\tilde{a}\rho \mu(\varphi)$
where $\mu(\varphi)=2(-\cos \nu \sin \varphi + \sin \nu \cos
\varphi)$.
Hence
\[
\geomfac{2}^\circ(\theta,\varphi_0,\v x)=(\tilde{a}\rho
\mu(\varphi_0(\theta,\v x)))^{-1}.
\]
Note that dependence on $\theta$ is hidden in
$\tilde{a}=a(\theta)\sin(\theta)$ and $\rho=\sqrt{x^2+y^2}$.

Now, we first pick the root with the positive imaginary part, and write
$\varphi_0(\theta,\v x)=\nu+i\eta$, where $\eta=\ln \beta=
\ln \left(\lambda+\sqrt{\lambda^2-1}\right)$. We have $\lambda>1$ and
hence $\beta>1$ and $\eta>0$. 
With this, we have
\[
\mu(\varphi_0)=2(-\cos \nu \sin(\nu +i \eta)+\sin \nu \cos(\nu +i
\eta))=\frac{1}{i}\left( e^\eta -e^{-\eta} \right).
\]
For the root with the negative imaginary part, we get
$\mu(\varphi_0)=\left( e^{-\eta} -e^{\eta} \right)/i$. Since the
absolute value will be taken, this will yield the same result in the
end.  We can evaluate
\begin{align}
   E_{fac}^{TZ}(\v
  x,\theta) &=\abs{\geomfac{2}^\circ\pars{\theta,\varphi_0(\theta, \v x),\v x }^p
              \left( e^{- \abs{\Im \varphi_0(\theta,\v
                  x)}}\right)^{n_\varphi}} \notag \\ &=
            \frac{\left( e^{-
                  \eta}\right)^{n_\varphi}}{\left(\tilde{a}\rho
                \left( e^\eta -e^{-\eta} \right) \right)^p}
            =\frac{1}{\left(\tilde{a}\rho \right)^p}
    \frac{1}{\left( \beta - 1/\beta\right)^p}\frac{1}{\beta^{n_\varphi}}.
           \end{align}
Replacing $\beta=\lambda+\sqrt{\lambda^2-1}$ and noting that
$\beta-1/\beta=2\sqrt{\lambda^2-1}$
we obtain
\begin{equation}
   E_{fac}^{TZ}(\v x,\theta) =\frac{1}{\left(2 \tilde{a}\rho \right)^p}
   \frac{1}{\left( \sqrt{\lambda^2-1} \right)^p}
   \frac{1}
   {  \left( \lambda+\sqrt{\lambda^2-1} \right)^{n_\varphi}},
   \label{eq:EfacTZ_ver2}
\end{equation}
which can be rewritten as \cref{eq:EfacTZ}.

We have $\lambda>1$, and 
\begin{align}
\frac{\lambda}{\sqrt{\lambda^2-1}} <
  \frac{\lambda}{\sqrt{\lambda^2-2\lambda +1}}=
  \frac{\lambda}{\sqrt{(\lambda-1)^2}}=\frac{\lambda}{\lambda-1}
  < \frac{\lambda_0}{\lambda_0-1},
\end{align}
where the last step holds true for $\lambda>\lambda_0$ as this
quantity approaches $1$ from above as $\lambda$ increases. 
Similarly
\begin{align}
 \frac{1}{\lambda+\sqrt{\lambda^2-1}}
  <\frac{1}{\lambda+\sqrt{(\lambda-1)^2}}= \frac{1}{2\lambda-1} < \frac{2 \lambda_0}{2\lambda_0-1}\frac{1}{2\lambda}.
\end{align}
Using these inequalities in \cref{eq:EfacTZ}, the result follows. 
\end{proof}

\begin{remark}
$E_{fac}^{TZ}(\v x,\theta)$ uses the root $\varphi_0(\theta,\v x)$, and is
hence not defined for $\rho^2=x^2+y^2=0$ or $\theta=\{0,\pi\}$ (see 
the remark following \cref{lemma:phiroot_given_theta}). 
In the next theorem we bound the maximum value of $E_{fac}^{TZ}$ for a
spherical surface as
the evaluation point approaches the $z$-axis. 
\end{remark}
\begin{theorem}
Let the sphere $\gammaaxisym(\theta,\varphi)$ be
parametrized as in \cref{eq:gamma_axisymm} with
$a(\theta)=b(\theta)=a$, and let all else be defined as in 
\cref{thm:EfacTZ_est}.
Given an evaluation point $\v x =(x,y,z)\in \reals^3$ not on $\gammaaxisym$, 
set $m=\|\v x\|/a$, $\alpha=\atan(\sqrt{x^2+y^2},z)$ and
$\beta=\atan(y,x)$ such that 
\begin{equation}
  \v x=\v x(\alpha)= m a \left( \sin(\alpha) \cos(\beta),
    \sin(\alpha) \sin(\beta), \cos(\alpha) \right).
  \label{eqn:xalpha}
\end{equation}

It then follows
\[
\max_{\theta}E_{fac}^{TZ}(\v x(\alpha), \theta) =E_{fac}^{TZ}(\v x(\alpha), \alpha) 
\le
C \frac{m^{n_\varphi}}{a^{2p}}\left(\frac{1}{1-m}\right)^{2n_\varphi}
 \left( \sin^2 \alpha \right)^{n_\varphi},
\]
where the constant $C$ is the same as in \cref{eq:EfacBound}.
Note that the value of $\beta$ does not affect the result due to the
axisymmetry.

\label{thm:EfacTZmax_sph}
\end{theorem} 
\begin{proof} (\cref{thm:EfacTZmax_sph})
The point on the surface of the sphere closest to $\v x(\alpha)$ is
  the point with $\theta =\alpha$ at the same azimuthal angle
  ($\varphi=\beta$).  Differentiating
  $E_{fac}^{TZ}(\v x(\alpha), \theta)$ as given in \cref{eq:EfacTZ}
  with respect to $\theta$, one can show that $\theta=\alpha$ yields
  the maximum, as one would expect (the formulas do however get very
  long).
 
Starting with the formula in \cref{eq:EfacBound}, we can evaluate the
expression in the first denominator using \cref{eqn:xalpha} and $\theta=\alpha$,
\[
  \tilde{a}^2+\rho^2+(\tilde{b}-z)^2=
  a^2(1+m^2)\sin^2\alpha+a^2(1-m)^2\cos^2 \alpha=a^2(1-m)^2+2a^2m\sin^2\alpha,
\]
and with $\tilde{a}\rho=a^2m\sin^2\alpha$, 
\[
\lambda=\frac{\tilde{a}^2+\rho^2+(\tilde{b}-z)^2}{2 \tilde{a}\rho}=\frac{(1-m)^2}{2m}\frac{1}{\sin^2\alpha}+1.
  \]
Using that  $ \tilde{a}^2+\rho^2+(\tilde{b}-z)^2 \ge a^2(1-m)^2$ and
$\lambda > (1-m)^2/(2m\sin^2\alpha)$, the result
follows. 
\end{proof}
\begin{corollary}
Let all be as in \cref{thm:EfacTZmax_sph}, and specifically
  let the evaluation point $\v x(\alpha)$ be as in \cref{eqn:xalpha}
with $m\ne 1$. 
  As the evaluation point $x$ approaches the $z-axis$ the trapezoidal rule error 
 $E_{\v \gamma}^{TZ}(f,p,n_\varphi,\v x)$ in \cref{eq:ETZ2} vanishes,
 i.e. 
\[
\lim_{\alpha \rightarrow \{ 0,\pi \}}
E_{\v \gamma}^{TZ}(f,p,n_\varphi,\v x(\alpha))=0.
 \]
\label{corr:ErrTZpole}
\end{corollary} 
\begin{proof} (\cref{corr:ErrTZpole})
Using the definition in \cref{eq:ETZ2}, 
 \begin{align*}
  E_{\v \gamma}^{TZ}(f,p,n_\varphi,\v x)&=\frac{4\pi}{ \Gamma(p)}
                                       n^{p-1} \int_{E_1} 
  \abs{f\pars{t,\varphi_0(\theta(t),\v x)} } E_{fac}^{TZ}(\v
                                       x,\theta(t))  \dif t  \\
   & \le  \max_{\theta}  E_{fac}^{TZ}(\v
                                       x,\theta) \, \frac{4\pi}{ \Gamma(p)}
                                       n^{p-1} \int_{E_1} 
     \abs{f\pars{t,\varphi_0(\theta(t),\v x)} }  \dif t . 
 \end{align*}
 Hence with
 $C=\frac{4\pi}{ \Gamma(p)} n^{p-1}
 \int_{E_1} \abs{f\pars{t,\varphi_0(\theta(t),\v x)} }  \dif t $, it follows
 \[
 E_{\v \gamma}^{TZ}(f,p,n_\varphi,\v x(\alpha)) \le C E_{fac}^{TZ}(\v x(\alpha),\alpha).
 \]
 From \cref{thm:EfacTZmax_sph} it follows that
$ \lim_{\alpha \rightarrow \{ 0,\pi \}} E_{fac}^{TZ}(\v
x(\alpha),\alpha)=0$, and the result follows. 
\end{proof}

\begin{figure}[htbp]
  \centering
  \begin{subfigure}{.32\textwidth}
    \centering
    \includegraphics[width=\textwidth]{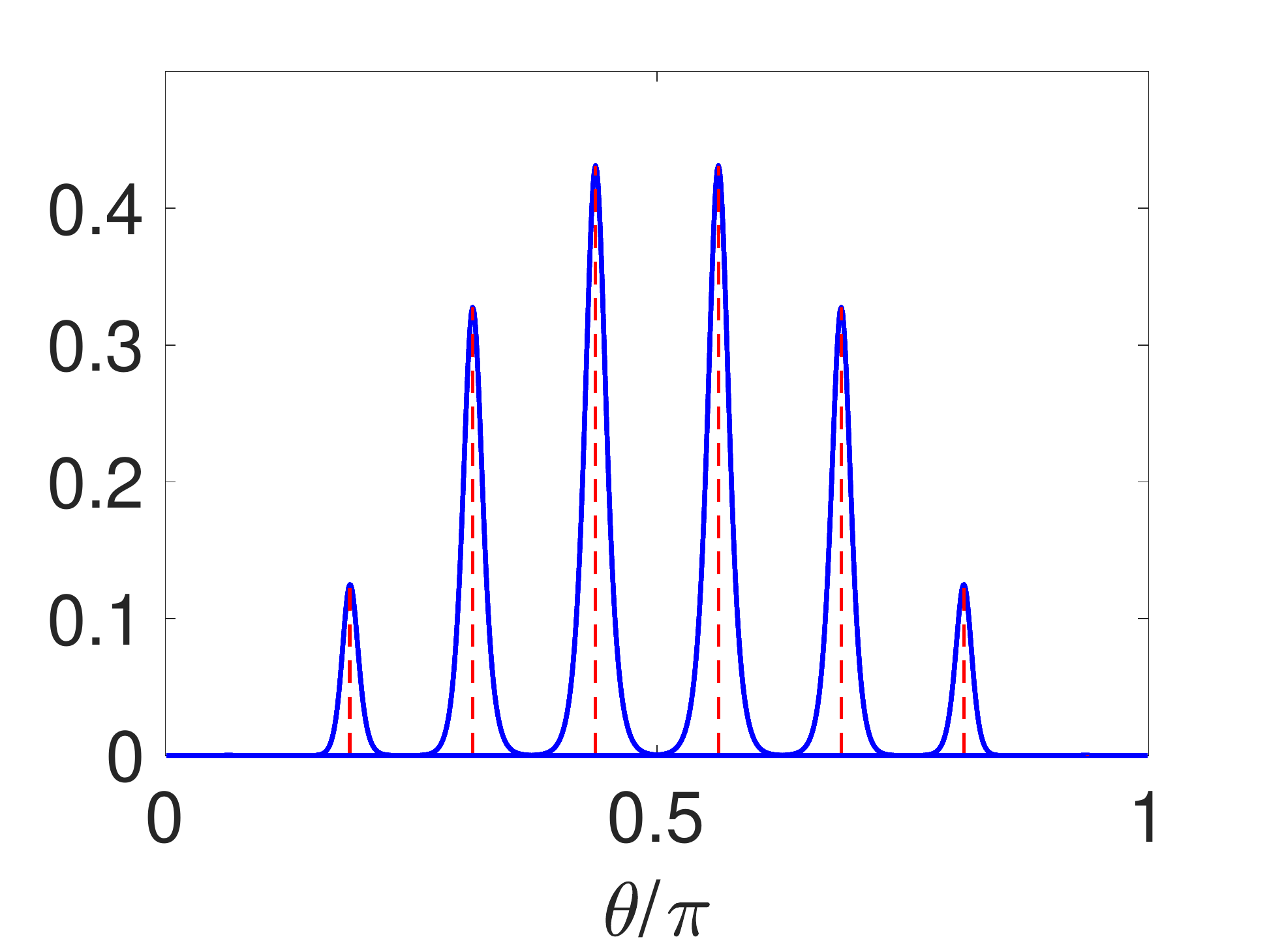}
    \caption{}
  \end{subfigure}
   \begin{subfigure}{.32\textwidth}
    \centering
    \includegraphics[width=1\textwidth]{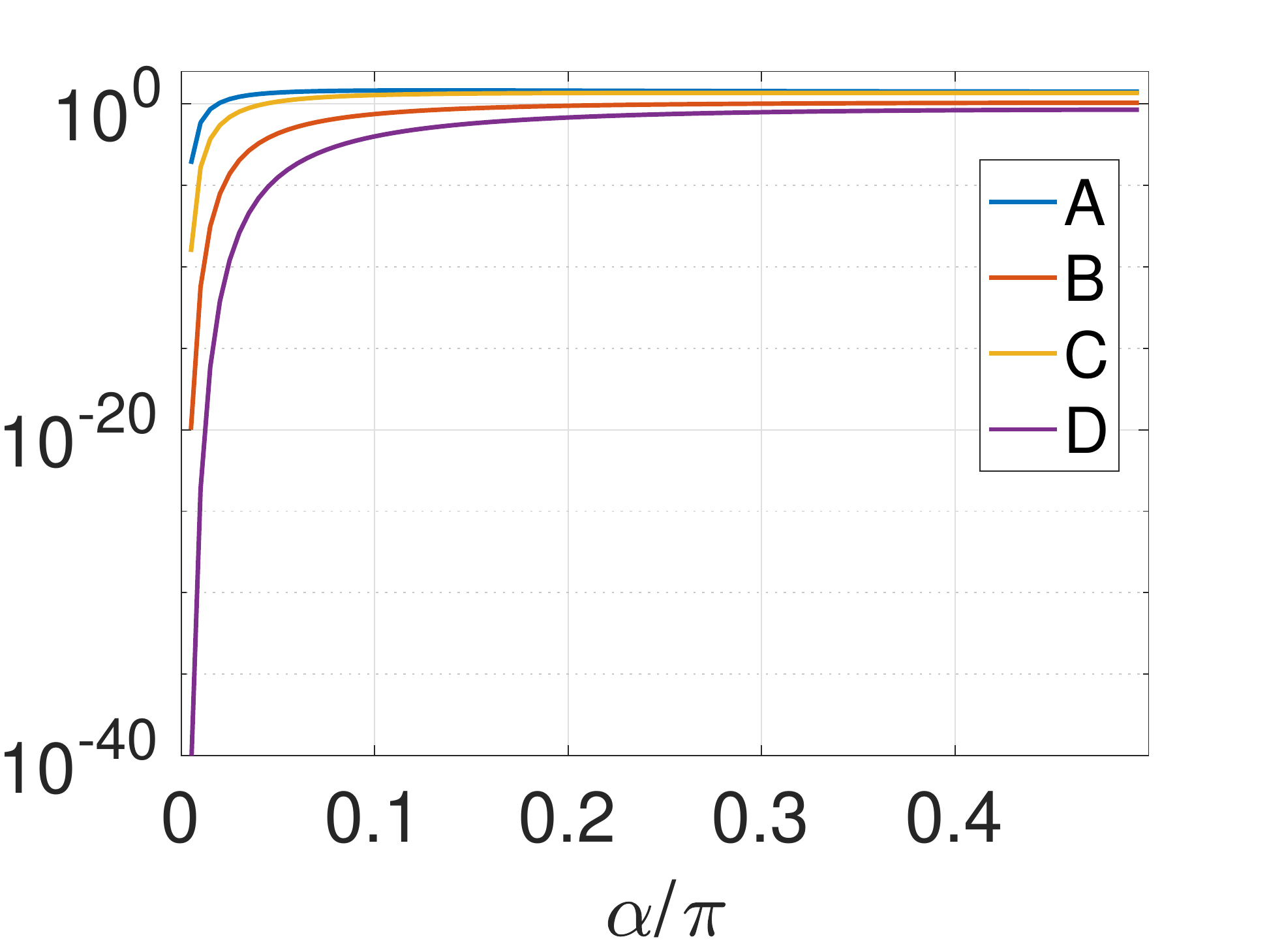}
    \caption{}
  \end{subfigure}
  \begin{subfigure}{.32\textwidth}
    \centering
    \includegraphics[width=1\textwidth]{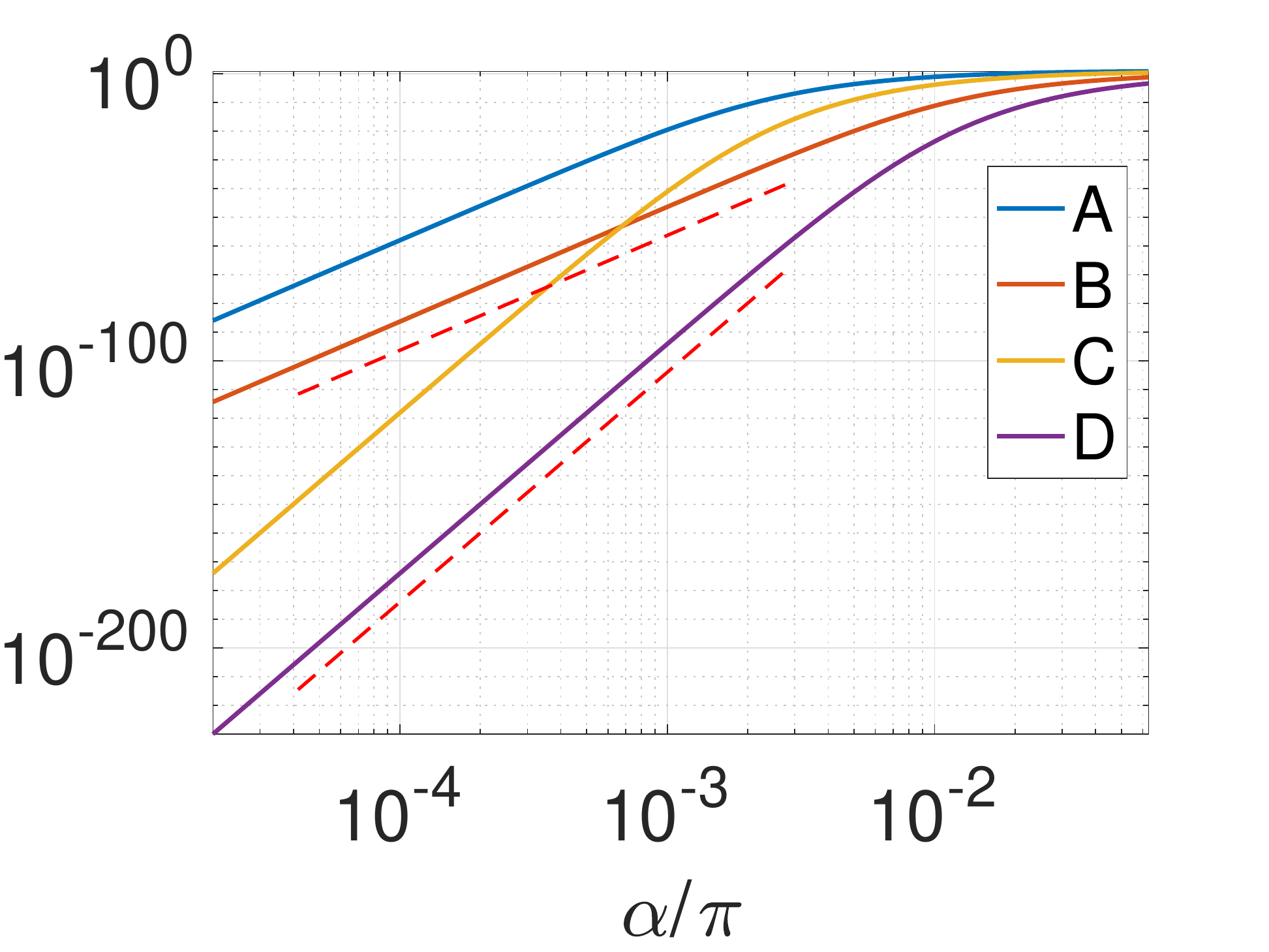}
    \caption{}
  \end{subfigure}
 \caption{
Illustration of results related to \cref{thm:EfacTZmax_sph} for a sphere of
radius $a=1$ with $p=1/2$. 
  (a) Plot of $E_{fac}^{TZ}(\v x(\alpha),\theta)$ versus
  $\theta/\pi$ for several values of $\alpha$ with $n_\varphi=20$ and $m=1.05$. Each curve
  peaks at $\theta=\alpha$, marked with a red vertical dashed
  line, and decays rapidly away from $\theta=\alpha$. (b) Plot of $E_{fac}^{TZ}(\v x(\alpha),\alpha)$ versus
  $\alpha/\pi$ for (A): $n_{\varphi}=20$, $m=1.01$, (B):
  $n_{\varphi}=20$, $m=1.05$, (C): $n_{\varphi}=40$, $m=1.01$, (D):
  $n_{\varphi}=40$, $m=1.05$. 
 (c) Same as (b) but with log scale over small values of $\alpha$. The
 two dashed red lines indicate the slopes $\alpha^{2n_\varphi}$
  for $n_{\varphi}=20$ and $40$, respectively. 
  }
\label{fig:EfacTZ}
\end{figure}

To the left in Fig. \ref{fig:EfacTZ}, we plot $E_{fac}^{TZ}(\v x(\alpha),\theta)$ versus
  $\theta/\pi$ for several values of $\alpha$ for a sphere of radius
  $1$, with $p=1/2$.  $E_{fac}^{TZ}(\v x(\alpha),\theta)$ 
  peaks at $\theta=\alpha$, and decays rapidly away from
  $\theta=\alpha$. The closer the evaluation point is to the surface
  (i.e. the closer $m$ is to $1$), the larger the maximum magnitude
  for a fixed number of discretization points $n_\varphi$. For a fixed
  $m$, the maximum magnitude decreases rapidly with $n_\varphi$. 
In the middle figure, we plot  $E_{fac}^{TZ}(\v x(\alpha),\alpha)$ versus
  $\alpha/\pi$ for the combinations of $n_\varphi=20$ and $40$ and
  $m=1.01$ and $1,05$. In the rightmost plot, we zoom in to see the
  behavior for small values of $\alpha$. As $\sin \alpha \approx
  \alpha$ for small $\alpha$, we expect to see a decay  proportional
  to $\alpha^{2n_\varphi}$, and we indicate these slopes in the plot
  for the two values of $n_\varphi$.

 From \cref{corr:ErrTZpole}, we have the result that the trapezoidal
 rule error $E_{\v \gamma}^{TZ}$ is bounded by a constant times $E_{fac}^{TZ}(\v
 x(\alpha),\alpha)$, and in the plots we can see the fast decay of $E_{fac}^{TZ}(\v
 x(\alpha),\alpha)$ with decreasing $\alpha$.
 
Let us now consider the Gauss-Legendre rule error. 
Write $E_{\v \gamma}^{GL}$
\cref{eq:EGLdef} from \cref{est:3D} as
 \begin{align}
E_{\v \gamma}^{GL}(f,p,n_t,\v x) &=  \frac{4\pi}{ \Gamma(p)}
                                (2n_t+1)^{p-1}
                                \int_{E_2} 
                                \abs{f\pars{t_0(\varphi, \v x),\varphi}}
                                E_{fac}^{GL}(\v x,\varphi) \dif \varphi ,
\label{eq:EGLdefwEfac}
\end{align}
where
\begin{align}
E_{fac}^{GL}(\v x,\varphi) =   \abs{\geomfac{1}\pars{t_0(\varphi),\varphi ,\v
  x}}^p
  \frac{ \abs{\sqrt{t_0(\varphi)^2-1}}^{1-p}} {
\abs{   t_0(\varphi) + \sqrt{t_0(\varphi)^2-1}     }^{2n+1}}.
\label{eq:EfacGL}
\end{align}

This term can be directly evaluated for points on the symmetry
axis. For an evaluation point $\v x=(0,0,z)$, we find that for a
sphere of radius $a$, 
\begin{align}
E_{fac}^{GL}(\v x,\varphi) =\frac{1}{2a\abs{z}}
\frac{1}{\abs{z^2-a^2}^{p-1}}
  \left(\frac{1}{\delta} \right)^{2n_t+1}
  \label{eqn:EfacGLzaxis}
\end{align}
where $\delta=\abs{z}/a$ if $\abs{z}>a$ and $\delta=a/\abs{z}$ if
  $\abs{z}<a$. We use the result in
  \cref{corr:thetaroot_given_phi_on_zaxis} for this derivation. For
  details, see \cref{app:GLders}. 
 This result is independent of $\varphi$, and hence we have
  \begin{align}
E_{\v \gamma}^{GL}(f,p,n_t,\v x) &=  \frac{4\pi}{ \Gamma(p)}
\frac{1}{2a\abs{z}}
\frac{(2n_t+1)^{p-1}}{\abs{z^2-a^2}^{p-1}}
\left(\frac{1}{\delta} \right)^{2n_t+1}
                                \int_{E_2} 
                                \abs{f\pars{t_0(\varphi, \v x),\varphi}}
                               \dif \varphi .
\label{eq:EGLSphSymmAx}
\end{align}
This means that the
 Gauss-Legendre error $E^{GL}$ will strongly dominate over the
 trapezoidal rule error $E^{TZ}$ for evaluation points close to the symmetry
 axis, and the trapezoidal rule error can hence safely be ignored.

For a general axisymmetric surface, we cannot follow the same approach
as for the sphere, were we could identify the two parametrization
angles for the closest point on the surface, and furthermore show that $E_{fac}^{TZ}(\v x,\theta)$
attains its maximum value for that value of the polar angle $\theta$.
As an evaluation point sufficiently close to a general axisymmetric surface approaches
the $z$-axis, the closest point on the surface will however be the north or
the south pole. Hence, it is interesting to investigate this limit. 
\begin{theorem}
Let $E_{fac}^{TZ}(\v x,\theta)$ be defined as in
\cref{thm:EfacTZ_est}, and assume $z$ of the evaluation point 
$\v x =(x,y,z)\in \reals^3$ such that
$(b(\theta)\cos(\theta)-z)^2 =(\tilde{b}-z)^2>0$ for  $\theta \in \{
[0,\beta]\cup [\pi-\beta, \pi]\}$, for some $\beta>0$. Then for this
range of $\theta$ it holds 
\begin{align}
  E_{fac}^{TZ}(\v x) \le \frac{C}{(\tilde{b}-z)^{2(p+n_{\varphi})}}
  (\rho a(\theta) \sin\theta)^{n_{\varphi}}.
\label{eq:EfacBound_arho}
\end{align} 
In the limit as $\rho \rightarrow 0$ or $\tilde{a}=a(\bar{\theta})
\sin(\bar{\theta}) \rightarrow 0$, or both, we have $ E_{fac}^{TZ}(\v x)
\rightarrow 0$.
\label{thm:EfacTZ_limit}
\end{theorem}
\begin{proof} (\cref{thm:EfacTZ_limit})
We start with \cref{eq:EfacBound} and rewrite $\lambda$ as 
$\lambda=(\tilde{a}/\rho+\rho/\tilde{a})/2+(\tilde{b}-z)^2/(2\tilde{a}
\rho)$. As it is assumed  $(\tilde{b}-z)^2>0$, we can use the simple
estimate
\[
\lambda \ge 1+\frac{(\tilde{b}-z)^2}{2 \tilde{a}\rho} >\frac{(\tilde{b}-z)^2}{2 \tilde{a}\rho},
\]
and hence
\begin{align}
  E_{fac}^{TZ}(\v x, \theta) < \frac{C}{(\tilde{a}^2+\rho^2+(\tilde{b}-z)^2)^p}
  \left(  \frac{\tilde{a}\rho}{(\tilde{b}-z)^2}\right)^{n_{\varphi}}
<\frac{C}{(\tilde{b}-z)^{2(p+n)}}
  \left( \tilde{a}\rho \right)^{n_{\varphi}},
  \label{eq:EfacBoundZZ}
\end{align}
which is the desired result. 
\end{proof}

For general axisymmetric surfaces, we have not been able to prove that the trapezoidal error
vanishes as the evaluation point approaches the $z$-axis, as we have
done for the sphere in \cref{corr:ErrTZpole}.
From the theorem above, we have the result that $E_{fac}^{TZ}(\v x,
\theta)$ vanishes as $\v x$ approaches the $z$-axis for some range of
$\theta$, under some assumption on the $z$-coordinate. We have however
not proven for which value of $\theta$ the maximum of $E_{fac}^{TZ}(\v x,
\theta)$ is attained. We do conjecture that this $\theta \rightarrow
\{0,\pi \}$ as the evaluation point approaches the $z$-axis for $z<0$
and $z>0$, respectively, as this $\theta$ is the parameter for the
point on the surface closest to the evaluation point.

Numerically, we note that the trapezoidal error contribution to the
total error estimate decays rapidly as the evaluation point approaches
the $z$-axis also for general axisymmetric surfaces. 

\subsection{Simplified error estimate for a spherical surface}
\label{sec:simplified_est}

In this section, we will present a simplified error estimate for a
spherical surface, and then discuss the derivation of it, starting
from \cref{est:3D}.
\begin{estimate}
Let $\gamma(t,\varphi)=\gammaaxisym(\theta(t),\varphi)$, where
$\gammaaxisym$ is the sphere
parametrized as in \cref{eq:gamma_axisymm} with
$a(\theta)=b(\theta)=a$, and $\theta(t)=\pi-\cos^{-1}(t)$. 
Consider the integral in \cref{eq:3Dpot} with $k \equiv 1$ and $\sigma
\equiv 1$ such that $f(t,\varphi)=\norm{\partial \v\gamma/\partial t
  \times \partial \v\gamma/\partial \varphi}=a^2$, with the evaluation
point $\v x =(x,y,z)\in \reals^3$ not on $\gamma$.

Introduce $\zeta=\|\v x\|=\sqrt{x^2+y^2+z^2}$ and an even integer $n$. 
The error in approximating 
the integral with the $n_t =n/2$ point
Gauss-Legendre rule in the $t$-direction and the $n_\varphi=n$-point trapezoidal
rule in the $\varphi$ direction can
be estimated as 
\begin{align}
E_{sphere}(\zeta,a,p,n)= \frac{8\pi}{ \Gamma(p)}  n^{p-1} \frac{n!!}{(n+1)!!}
  \frac{a^2}{\abs{\zeta^2-a^2}^p} \left( \frac{1}{\delta} \right)^n,
  \label{eqn:Esphere}
\end{align}
 where $\delta=\zeta/a$ if $\zeta>a$ and $\delta=a/\zeta$ if
  $\zeta<a$.
\label{est:Esphere}
\end{estimate}
\begin{remark}
Note that this error estimate only depends on the evaluation point through $\zeta=\|\v x\|$. This means that
the error is estimated to decay equally in all directions as we move
away from the sphere. This is only a good approximation under the
map $t=-\cos(\theta)$, and not under the linear map, as will be
discussed in section \ref{sec:numer-sphere}. 
  \end{remark}

To derive the simplified estimate, 
we will use the assumption that the error only depends on the distance
to the sphere, and will pick an evaluation point that yields the
simplest expressions to work with.
We will start by considering the trapezoidal rule error at a point $\v x =(x,y,0)$; in this case $\zeta$ corresponds to the distance from the z-axis that was previously denoted by $\rho$, so we will continue with this notation, assuming
$\rho^2=x^2+y^2\ne a$.
The expression for  $E_{fac}^{TZ}(\v x,\theta)$ is given in
\cref{eq:EfacTZ}, but we will start with the equivalent expression in
\cref{eq:EfacTZ_ver2}.
For this case we obtain
  $\lambda=(\rho/a+a/\rho)/ (2 \sin \theta)=(\delta+1/\delta)/ (2 \sin \theta)$
  where we let $\delta=\rho/a$ if $\rho>a$ and $\delta=a/\rho$ if
  $\rho<a$, such that $\delta>1$. With this, we have
  \begin{align}
\lambda^2-1=\left(\frac{1}{2 \sin \theta} \left(\delta-\frac{1}{\delta}
  \right) \right)^2+\frac{\cos^2\theta}{\sin^2 \theta}.
\label{eqn:lambdasq_m1_tz}
  \end{align}
From \cref{thm:EfacTZmax_sph}, we know that $E_{fac}^{TZ}(\v
x,\theta)$ will attain its maximum at $\theta=\pi/2$ for the chosen
evaluation point. At this value of $\theta$, the last term vanishes. 
Ignoring that term we get $\lambda+\sqrt{\lambda^2-1} \approx \delta/\sin
\theta$ with equality at $\theta=\pi/2$. If we use this approximation, and 
evaluate the part of $E_{fac}^{TZ}$ that is taken to the power of $p$
at $\theta=\pi/2$, we obtain
\[
 E_{fac}^{TZ}(\v x,\theta) \approx
 \frac{1}{\abs{\rho^2-a^2}^p}\left(\frac{\sin \theta}{\delta} \right)^n.
  \]
Now, under the map $t=-\cos(\theta)$ an approximation to the
trapezoidal rule error will be
\begin{align}
  E _{\v \gamma}^{TZ}(a^2,p,n,\v x) & \approx
  \frac{4\pi}{ \Gamma(p)}  n^{p-1}  \frac{1}{\abs{\rho^2-a^2}^p}
 \int_{-1}^{1} \left(\frac{\sin \theta(t)}{\delta} \right)^n \, dt \\
 &= \frac{4\pi}{ \Gamma(p)}  n^{p-1}  \frac{1}{\abs{\rho^2-a^2}^p}
  \left(\frac{1}{\delta} \right)^n\int_{-1}^{1} \left( \sqrt{1-t^2} \right)^n\, dt.
\end{align}
Under the coordinate transformation $t=\sin \beta$, the integral can be written as
\begin{align}
\int_{-\pi/2}^{\pi/2} (\cos \beta)^{n+1} \, d\beta.
\label{eqn:cosbetaint}
\end{align}
Using that for
$q>1$,
\[
\int_{-\pi/2}^{\pi/2} (\cos \beta)^{q} \, d\beta=\frac{q-1}{q} \int_{-\pi/2}^{\pi/2} (\cos \beta)^{q-2} \, d\beta,
  \]
  we obtain
  \[
\int_{-1}^{1} \left(1-t^2\right)^{n/2}\, dt=2
\frac{n(n-2)\cdots 2}{(n+1)(n-1)\cdots 3}=\frac{2n!!}{(n+1)!}
\]
and in total we get 
\begin{align}
  E _{\v \gamma}^{TZ}(a^2,p,n,\v x)  \approx
  \frac{8\pi}{ \Gamma(p)}  n^{p-1} \frac{n!!}{(n+1)!} \frac{a^2}{\abs{\rho^2-a^2}^p}
  \left(\frac{1}{\delta} \right)^n.
\label{eqn:sph_ETZ_approx}
\end{align}

We will now continue by estimating the size of the Gauss-Legendre rule
error. 
We write $E _{\v \gamma}^{GL}$ from \cref{est:3D} as in \cref{eq:EGLdefwEfac}, with
$E_{fac}^{GL}(\v x,\varphi)$ defined in \cref{eq:EfacGL}.

Again, picking an evaluation point at the equator, we can derive the approximation
\begin{align}
  E _{\v \gamma}^{GL}(a^2,p,n,\v x)  \approx
  \frac{8\pi}{ \Gamma(p)}  (n+1)^{p-1} \frac{n!!}{(n+1)!} \frac{a^2(\rho/a+a/\rho)}{\abs{\rho^2-a^2}^p}
  \left(\frac{1}{\delta} \right)^{n+1},
\label{eqn:sph_EGL_approx}
\end{align}
where $\rho=\|\v x\|$. The details of this derivation is given in
\cref{app:GLders}. 

Comparing  \cref{eqn:sph_ETZ_approx} ($n=n_\varphi$) and
\cref{eqn:sph_EGL_approx} ($n=2n_t$), the expressions are very
similar.
The ratio is $E _{\v \gamma}^{GL}/E _{\v \gamma}^{TZ}=((n+1)/n)^{p-1}(1+1/\delta^2)$.

This means that the contributions of the two errors are of about equal size for
evaluation points in the $xy$-plane at $z=0$, as opposed to
the case of evaluation points at the $z$-axis, where the contribution
for the trapezoidal rule  $E _{\v \gamma}^{TZ}$ vanishes. The total error is however
approximately equal for the same $\zeta=\|\v x\|$.
Since there is some overestimation of the errors, we have chosen to
estimate the full error as a function of $\zeta$ by the derived
expression for the trapezoidal rule error at the equator, as is given in
\eqref{eqn:Esphere} in \cref{est:Esphere}. The accuracy of this estimate
will be numerically evaluated in \cref{sec:numer-sphere}.

\section{Numerical evaluation of the error estimate} 
\label{sec:numeval}
An estimate $\Eest$ for the quadrature error $\EQQ$ in the evaluation of a layer
potential was given in \cref{est:3D}. The integrals in
\eqref{eq:ETZdef}-\eqref{eq:EGLdef} adding up to $\Eest$ can however not be evaluated
analytically, and we need to introduce an approximation that is
sufficiently precise and computationally cheap to evaluate. For
general surfaces, we do  not have access to analytical expressions for
the roots appearing in the estimate. Hence, they must be computed
using a numerical root finding procedure, and it will be of interest
to minimize the number of root evaluations.

The integrand in \eqref{eq:ETZdef} is not well defined for evaluation
points along the symmetry axis for an axisymmetric surface, as was
discussed in section \ref{sec:estsymmax}. We proved that the
contribution from the trapezoidal error \eqref{eq:ETZdef} vanishes in
the limit of the target point approaching the symmetry axis of a
sphere (Corollary \ref{corr:ErrTZpole}).
For a general axisymmetric surface, we were able to prove only a
weaker result but we
quantified the contribution of the trapezoidal error \eqref{eq:ETZdef} compared to the
contribution of the Gauss-Legendre error \eqref{eq:EGLdef} by
numerical experiments. 
Also here we find that the first quantity decays very rapidly as the
evaluation point approaches the z-axis, depending also on the distance
of the evaluation point from the surface. To be more precise, the problematic target points lie in the region
where $\rho \rightarrow 0$ or/and
$a(\theta)\sin(\theta) \rightarrow 0$ (see Theorem
\ref{thm:EfacTZ_limit}) geometrically represented by the cones with
apices at the poles and increasing width inside or outside the surface
respectively for the interior or exterior problem. To see why, consider a number of target points placed in the normal direction starting from a grid point $\gamma(t^*,\varphi^*)$; it is clear that they will all refer to the same closest grid point, but $\rho$ (the distance from the z-axis) will increase together with the distance from the surface.
In practical
applications we will safely ignore the trapezoidal rule
contribution for the evaluation points $\v x$ such that
\begin{equation}
\label{eq:cdt_cone}
\frac{\rho}{A}<\frac{K_c \pi}{n_t}\min_{(t,\varphi)} \norm{\v \gamma(t,\varphi) - \v x},
\end{equation} 
where $K_c$ is a fixed constant, $\rho$ the distance to the $z$-axis
and $A$ a representative radius/length scale. This procedure is not
very sensitive to the choice of $A$ and $K_c$, as there is a rather
wide range where the trapezoidal rule error contribution is
negligible but it is still numerically stable to compute the error estimate. 
Practically, to evaluate the distance between the target
point and the surface it is sufficient to approximate the minimum in
\eqref{eq:cdt_cone} by the minimum over the surface grid points only. 

When considering non-axisymmetric surfaces, the situation is much less
predictable and we need a different strategy. This is based on using a
local approximation of the surface centered away from the poles; in
this way all the quantities needed for the estimate evaluation are
locally well defined, and the singularities mentioned above are
eluded. How to numerically compute such an approximation will be further discussed in the next subsections, where we describe how to practically evaluate
the two error contributions \eqref{eq:ETZdef} and
\eqref{eq:EGLdef}, that add up to the total error. We start by the
approximation of the integrals before we discuss the root finding.

\subsection{Approximation of integrals in the error estimate}
\label{sec:errest}
In this section, we discuss how to approximate the integrals in
\eqref{eq:ETZdef}-\eqref{eq:EGLdef}, to be able to efficiently compute
a sufficiently precise estimate $\Eest$ for the quadrature error $\EQQ$ as defined in \cref{est:3D}.

Given an evaluation point $\v x \in \reals^3$, we start by
identifying $\v x^*\in \reals^3$, the closest discrete point on the surface
$\v \gamma$, and the parameters $t^*$, $\varphi^*$ such that $\v \gamma(t^*,
\varphi^*)=\v x^*$. This means that $t^*$ will be one of the $n_t$
Gauss-Legendre quadrature nodes, and $\varphi^*$ will be one of the
$n_\varphi$ (equidistant) trapezoidal rule quadrature nodes.
Loosely speaking, the contribution to the quadrature error will have a
peak around $\v x^*$. What this means is that
$\est^{TZ}(\varphi_0(t, \v x),n_{\varphi},p)$ in the integral over
$t$ in \eqref{eq:ETZdef} will have a peak close to $t=t^*$, and decay
rapidly away from $t^*$ due to the variation in $\varphi_0(t,\v
x)$. Similarly $\est^{GL}(t_0(\varphi,\v x),n_t,p)$ will have a peak
close to $\varphi=\varphi^*$, decaying rapidly away from  $\varphi^*$ due to the
variation in $t_0(\varphi,\v x)$.
Let us here remind that the roots $t_0(\varphi,\v x)$ and $\varphi_0(t,\v x)$ are roots
to $R^2(t,\varphi,\v x)$ with
one variable kept fixed, as defined in \cref{est:3D}, and further
denote 
\begin{align}
  t_0^*=t_0(\varphi^*,\v x), \quad \varphi_0^*=\varphi_0(t^*,\v x).
  \label{eqn:def_starroots}
\end{align}

We now assume that $\est^{TZ}$ is the most rapidly varying factor
in the integrand in \eqref{eq:ETZdef}, and similarly for $\est^{GL}$
in \eqref{eq:EGLdef}, and approximate 
\begin{align}
  E _{\v \gamma}^{TZ}(f,p,n_\varphi,\v x)& \approx
\abs{f\pars{t^*,\varphi_0^*} \geomfac{2}\pars{t^*,\varphi_0^* ,\v x}^p}
 \int_{E_1}  \est^{TZ}(\varphi_0(t, \v x),n_{\varphi},p) \dif t ,
 \label{eq:ETZ_GFconst}\\
  E _{\v \gamma}^{GL}(f,p,n_t,\v x) & \approx
  \abs{f\pars{t_0^*,\varphi^*} \geomfac{1}\pars{t_0^*,\varphi^* ,\v x}^p}
\int_{E_2}   \est^{GL}(t_0(\varphi,\v x),n_t,p) \dif \varphi. 
    \label{eq:EGL_GFconst}
\end{align}
This is typically a good approximation, unless the surface grid is
very stretched such that the grid resolutions on the surface in the two directions
are very different, in which case the geometry factors can vary
rapidly as well.

Now, we want to find a simple expression for how $t_0(\varphi,\v x)$
varies with $\varphi$ around $\varphi=\varphi^*$.
If we replace $\v \gamma$ with its bivariate linear approximation
around $(t^*,\varphi^*)$ in the 
definition of the squared distance function \eqref{eq:R2_def_surface},
we get a quadratic equation that we can solve to find the root. From
here, we find the approximation
\begin{align}
  t_0^L(\varphi) = t^* - \frac{b}{2c} 
  \pm i \frac{\sqrt{4ac-b^2}}{2c} ,
  \label{eq:t0_linear}
\end{align}
where, with $\Delta \varphi=\varphi-\varphi^*$, and $\v r=\v x^* - \v
x=\v \gamma(t^*,\varphi^*)-\v x$,
\begin{align*}
a=a(\Delta \varphi)&= \norm{\v r}^2   + 2 \pars{\v r \cdot \partial_\varphi\v\gamma(t^*,\varphi^*)}\Delta \varphi  + \norm{\partial_\varphi\v\gamma(t^*,\varphi^*)}^2 \Delta \varphi^2, \\
b=b(\Delta \varphi)&=2 \pars{ \v r \cdot \partial_t\v\gamma(t^*,\varphi^*)}  + 2 \pars{\partial_\varphi\v\gamma(t^*,\varphi^*)\cdot \partial_t\v\gamma(t^*,\varphi^*)}\Delta \varphi ,\\
c=c(\Delta \varphi)&= \norm{\partial_t\v\gamma(t^*,\varphi^*)}^2.
\end{align*}
The root $t_0^*$ is by definition the root at
$\varphi=\varphi^*$, and in practice, as will be discussed in the next
subsection, at least an accurate approximation to it, while
$t_0^L(\varphi^*)$ is a simpler approximation. In order to better capture the
magnitude of the peak of  $\est^{GL}(t_0(\varphi,\v x),n_t,p)$, we want to use this more accurate value, but we also want
to use the simple dependence on $\varphi$. This leads us to define
\begin{align}
  \tilde t_0(\varphi) = t_0^*  - t_0^L(\varphi^*) + t_0^L(\varphi).
  \label{eq:combinedapprox}
\end{align}
See \citep{AFKLINTEBERG20221} for more details and a discussion
regarding the effect of making these approximations.

The same approximations can naturally be made to define $\tilde
\varphi_0(t)$,  an approximation to $\varphi_0(t,\v x)$.
Away from $t=t^*$, we have
$\abs{Im \, \tilde \varphi_0(t)} \sim k~\abs{t-t^*}$, where 
$k=\norm{\partial_t\v\gamma(t^*,\varphi^*)}/\norm{\partial_\varphi\v\gamma(t^*,\varphi^*)}$,
and hence $\abs{Im \, \tilde t_0(\varphi)} \sim
k^{-1}~\abs{\varphi-\varphi^*}$. 
This means that $\est^{TZ}(\tilde \varphi_0(t),n_{\varphi},p)$ decays as
\begin{align}
e^{-n_\varphi\abs{Im \, \tilde \varphi_0(t)}} \sim e^{-n_\varphi k~\abs{t-t^*}} .
  \label{eq:trapz_decay}
\end{align}
Given this decay, it is a reasonable approximation to expand 
the interval of integration in \eqref{eq:ETZ_GFconst} from
$[-1,1]$ to $[-\infty,\infty]$, as the tails will be negligible, 
and we approximate the
integral in \eqref{eq:ETZ_GFconst} by
\begin{align}
  \int_0^\infty \est^{TZ}\pars{\tilde \varphi_0(t^* - s),n_\varphi,p} \dif s
  + \int_0^\infty \est^{TZ}\pars{\tilde \varphi_0(t^* + s),n_\varphi,p} \dif s.
  \label{eq:inf_expand}
\end{align}
With the variable transformation $x=n_\varphi ks$, we can write
\begin{align}
  \int_0^\infty
  \est^{TZ}\pars{\tilde \varphi_0(t^*\pm s)} \dif s  
  &=
    \frac{1}{n_\varphi k}
  \int_0^\infty
h^{TZ}_\pm(x)
    e^{-x}  \dif x ,
\end{align}
where $h^{TZ}_\pm(x)=\est^{TZ}\pars{\tilde \varphi_0(t^*\pm x/(n_\varphi k))}   e^{x}$.
Gauss-Laguerre quadrature is a Gaussian quadrature for integrals of
this type \cite[\S3.5(v)]{NIST:DLMF}, and we find that is is sufficiently accurate to
evaluate each of the integrals of $h^{TZ}_+(x)$ and $h^{TZ}_-(x)$ with 
8 quadrature nodes. 

For the Gauss-Legendre estimate, we have the bound
\cite{AfKlinteberg2016quad}
\begin{align}
 \est^{GL}(t_0,n,p) & =
  \frac{4\pi}{ \Gamma(p)}   \abs{ \frac{2n+1}{\sqrt{t_0^2-1}} }^{p-1}
\abs{t_0 + \sqrt{t_0^2-1}}^{-(2n+1)} \notag \\ &\leq  \frac{4\pi}{
                      \Gamma(p)}(2n)^{p-1}e^{-2n\abs{\Im t_0}}. 
\end{align}
Hence, with $t_0(\varphi,\v x)$ approximated with $\tilde
t_0(\varphi)$, we have an estimated decay $e^{-2n_t
k^{-1}~\abs{\varphi-\varphi^*}}$. 
Based on this decay we use the variable transformation
$x=2n_tk^{-1}s$, and write the approximation of the integral in \eqref{eq:EGL_GFconst} 
as
\begin{align}
  \frac{k}{2n_t}
  \left[ \int_0^\infty h^{GL}_-(x)  e^{-x}  \dif x + \int_0^\infty h^{GL}_+(x)  e^{-x} \dif x
  \right],
\end{align}
where  $h^{GL}_\pm(x)=\est^{GL}\pars{\tilde t_0(\varphi^*\pm xk/(2n_t)),n_t,p}   e^{x}$.
Again, each of these integrals is approximated with an $8$-point
Gauss-Laguerre quadrature rule. 

In \citep{AFKLINTEBERG20221}, we used this strategy for the global
trapezoidal rule discretization, and a different strategy for the panel
based Gauss-Legendre quadrature. Here, we have a mix of the two
quadrature rules, but both are used globally on the surface, and we have extended
this approach to be used for both contributions to the error
estimate.
It remains now to discuss the root finding, and the evaluation of the
factors in front of the integrals in 
\eqref{eq:ETZ_GFconst}-\eqref{eq:EGL_GFconst}.

\subsection{Root finding}
\label{sec:rootfinding}
To evaluate the error estimate as described in the previous
subsection, we need to determine
 $t_0^*=t_0(\varphi^*,\v x)$ and $\varphi_0^*=\varphi_0(t^*,\v x)$ as
 defined in \eqref{eqn:def_starroots}. For spherical topologies, these roots are not always well defined, as introduced in \ref{sec:estsymmax} for axisymmetric geometries (where $\varphi_0(\theta,\v x)$ is not defined for evaluation points on the symmetry axis) and further discussed at the beginning of section \ref{sec:numeval} for general surfaces. For axisymmetric geometries we identified the problematic region in the cone defined by eq. (\ref{eq:cdt_cone}): for these evaluation points we can ignore the trapezoidal error contribution and then we do not need to compute the roots $\varphi_0(\theta,\v x)$. We will then not consider these points in the following discussion. For a general surface, it is not so easy the determine a similar set, and we will proceed with a discrete approach as later discussed.

We will consider the $(\theta,\varphi)$ coordinate system, as also used in
  \eqref{eq:gamma_sph_harm} in the definition of a generic surface
  $\gamma^{\circ}(\theta,\varphi)$. In   \eqref{eq:gammast}, we define
$\gamma(t,\varphi)=\gamma^{\circ}(\theta(t),\varphi)$. Hence, once a
root $\theta_0$ has been determined, $t_0$ can be found using the
inverse map from $\theta$ to $t$. 
 
In \cref{sec:analytical}, we derived analytical expressions for the
roots of the squared distance function for special geometries: 
we have analytical expressions for $\theta_0$ only for a spherical
surface, and for $\varphi_0$ for any axisymmetric surface.
Hence, in general we need a numerical procedure to determine the
roots, and we will define this procedure in the $(\theta,\varphi)$
coordinate system.
Given an evaluation point $\v x=(x,y,z)$, we define
\begin{align}
 R^2(\theta,\varphi,\v x)=(\gamma^0_1(\theta,\varphi)-x)^2+(\gamma^0_2(\theta,\varphi)-y)^2+(\gamma^0_3(\theta,\varphi)
  -z)^2.
  \label{eq:R2circ}
\end{align}
Given a parametrization
$\v \gamma^{\circ}(\theta,\varphi)$, it is easy to solve
$R^2(\theta,\varphi,\v x)=0$ using a one dimensional Newton's method
to find a root $\varphi_0$ given $\theta$, or similarly, to
find a root $\theta_0$ given $\varphi$.
Specifically, in our setting, we need to determine 
$\theta_0^*=\theta_0(\varphi^*,\v x)$ and
$\varphi_0^*=\varphi_0(\theta^*,\v x)$, where $\theta^*=\theta(t^*)$. 
We typically use an  initial guess of $\theta^*+i/10$ for $\theta_0^*$, and
correspondingly for $\varphi_0^*$.
The iterations then usually converge with a strict tolerance in less
than five iterations. However in rare cases, usually for evaluation
points far away from the surface, it may happen that the iterations
fail to converge,
but most of the times it is sufficient to increase the magnitude of the
imaginary part of the initial guess for the iterations to converge
well. 

If no parametrization is available and we know only the quadrature
node values of $\v \gamma^{\circ}$ at the $n_t \times n_\varphi$ nodes,
we need to define an approximation
$\v \tilde \gamma^{\circ}(\theta,\varphi)$ that can be evaluated at
different arguments of $\theta$ and $\varphi$, 
and that allows for differentiation to formulate Newton's method.
One way is to compute the spherical harmonics
coefficients using a discrete transform \cite{Mohlenkamp99}.
Evaluating the spherical harmonics expansion is however a global procedure with
$O(n_t \, n_\varphi)$ cost for arbitrary arguments, and this would
need to be done at each step in the Newton iteration.

We can however use a local approach. 
In the Newton iteration, one of the variables, $\theta$ or $\varphi$
will be fixed. Assume that $\varphi=\varphi^*$ and introduce a 
$q$th order Taylor expansion in $\theta$, around $\theta=\theta^*$, 
\begin{align}
  \v{\tilde\gamma^\circ}(\theta,\varphi^*) = 
  \sum_{j=0}^q   
  \frac{ \pars{\theta-\theta^*}^j }{j!}
  \dpd[j]{\v{\gamma^\circ}}{\theta} (\theta^*,\varphi^*).
  \label{eq:gamma_taylor_3d_theta}
\end{align}
This approximation can be used in Newton's method to determine an
approximation to the
root $\theta_0^*=\theta_0(\varphi^*,\v x)$. Similarly, a Taylor
expansion in $\varphi$ can be introduced to obtain an approximation to
$\varphi_0^*=\varphi_0(\theta^*,\v x)$.
This is a solid strategy that eludes the pole singularities for any
kind of geometry. Indeed, since the discretization is based on
Gauss-Legendre nodes in the polar angle, the closest grid point where
the expansion \eqref{eq:gamma_taylor_3d_theta} is centered, will never
be a pole, avoiding the above mentioned problems.

For this reason, we will use this approach for non-axisymmetric
surfaces, even if we have a parametrization of $\v \gamma^\circ$
available. In this case, the analytical expression for
$\v \gamma^\circ$ can however be used to determine the $q$ first
partial derivatives of $\v \gamma^\circ$.  When this is not the case,
the derivatives both with respect to $\theta$ and $\phi$, can be
evaluated e.g from a spherical harmonics expansion. The advantage
compared to using a global expansion is that we can evaluate these
derivatives at all grid points in one sweep
\cite{Sorgentone2018167,Schaeffer}, and then use different
local expansions when estimating the quadrature error for different
evaluation points.

Once the roots have been determined, these Taylor expansions can also
be used to evaluate the geometry factors in
\eqref{eq:ETZ_GFconst}-\eqref{eq:EGL_GFconst} as defined in
\eqref{eq:geometry_factor_one}-\eqref{eq:geometry_factor_two}.
The roots $t_0^*$ and $\varphi_0^*$ are also needed to evaluate $f$ in
\eqref{eq:ETZ_GFconst}-\eqref{eq:EGL_GFconst}. We recall that $f$
depends on the density $\sigma$, which may not be known analytically
but be available at the grid points only (e.g. if $\sigma$ is a
solution to a discretized integral equation). In this case, again, $f$ can
either be approximated locally by a Taylor expansion, or globally by a
spherical harmonics expansion.

\newpage

\section{Numerical experiments}
\label{sec:numer-exper}
In this section, we will compare the quadrature error estimate $\Eest$
defined in \eqref{eq:ETZplusEGL}, and evaluated as discussed in the
previous section, to the measured error $\EQQ$ as defined in
\eqref{eq:EQdef} for some different examples. The measured error $\EQQ$ will be computed by using a reference solution on an upsampled grid with upsampling rate set to five. 

As outlined in Remark
\ref{rem:n_inf}, we will see that the estimate provides a good
approximation of the error also for moderate values of $n_t$ and
$n_\varphi$ as long as the geometry and the layer density are well
resolved. We choose $n_t$ and $n_\varphi$ so that this is true for all
the following numerical examples. For the rootfinding procedure, we
follow the strategy presented in the previous section: the analytical
expression for $\v\gamma\circ$ is used when dealing with axisymmetric
surfaces (excluding the trapezoidal error contribution for target
points defined in \eqref{eq:cdt_cone}), and $\v\gamma^\circ$ is
locally approximated with a Taylor expansion for other geometries. In
the first case, the parameters defining the cone in
\eqref{eq:cdt_cone} will be kept fixed to $A=1$ and $K_C=10$. In the
latter case, the order of the Taylor expansion will be fixed as $q=4$
(see eq. \eqref{eq:gamma_taylor_3d_theta}), which is accurate enough
for our purposes; an analysis of how the choice of $q$ can affect the
accuracy of the roots can be found in \citep{AFKLINTEBERG20221}. In
all the presented numerical tests we will use the analytical
expression for the density.



\subsection{A sphere}
\label{sec:numer-sphere}
In the first example we consider the harmonic single layer potential
\begin{equation}
u(\v x) = \int_{\Surf} \frac{\sigma(\v y)
  }{\norm{\v y - \v x}} \dif S(\v y),
  \label{eq:harmonic_sl}
\end{equation} 
evaluated near a sphere of radius $a=1$ with unit density,
$\sigma(\bf{x})=1$. We want to compare the estimated error $\Eest$ to the
actual measured error $\EQQ$, using the full error estimate in 
\eqref{eq:ETZplusEGL}, approximated as described in the previous
section, and, for the cosine map, also the simplified error estimate \eqref{est:Esphere},
derived in Section \ref{sec:simplified_est}.
For this case, referring to eq. \eqref{eq:ETZplusEGL}, $p=1/2$ and $f$ simplifies to
$$f(t,\varphi)= \norm{\dpd{\v\gamma}{t} \times \dpd{\v\gamma}{\varphi}}=\begin{cases}
\frac{\pi}{2}\sin((t+1)\frac{\pi}{2}), \text{ if using the linear mapping \eqref{eq:linear_map}}\\
1, \text{ if using the cosine mapping \eqref{eq:cos_map}.}
\end{cases}$$

Fig. \ref{fig:fig1} shows the resulting surface grids and the
different behavior of the quadrature error exterior to the sphere
with discretizations using the linear and non-linear mapping
$\theta(t)$ as given in \eqref{eq:linear_map}-\eqref{eq:cos_map}, with
$n_t=30$ and $n_\varphi=60$ points.
\begin{figure}[htbp]
  \centering
  \begin{subfigure}{.4\textwidth}
    \centering
    \includegraphics[width=\textwidth]{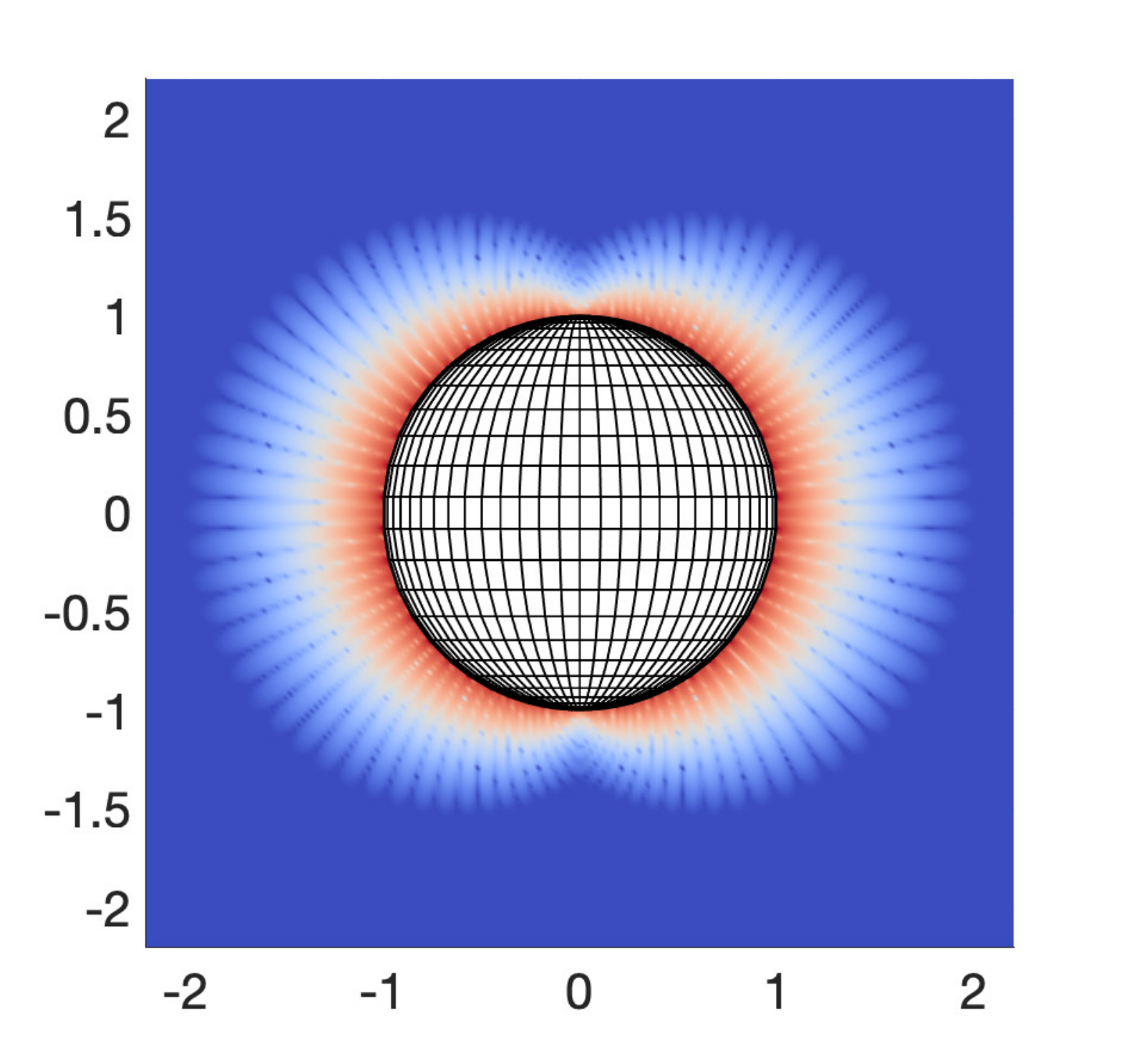}
    \caption{}
    \label{fig:fig1a}
  \end{subfigure}
   \begin{subfigure}{.4\textwidth}
    \centering
    \includegraphics[width=1\textwidth]{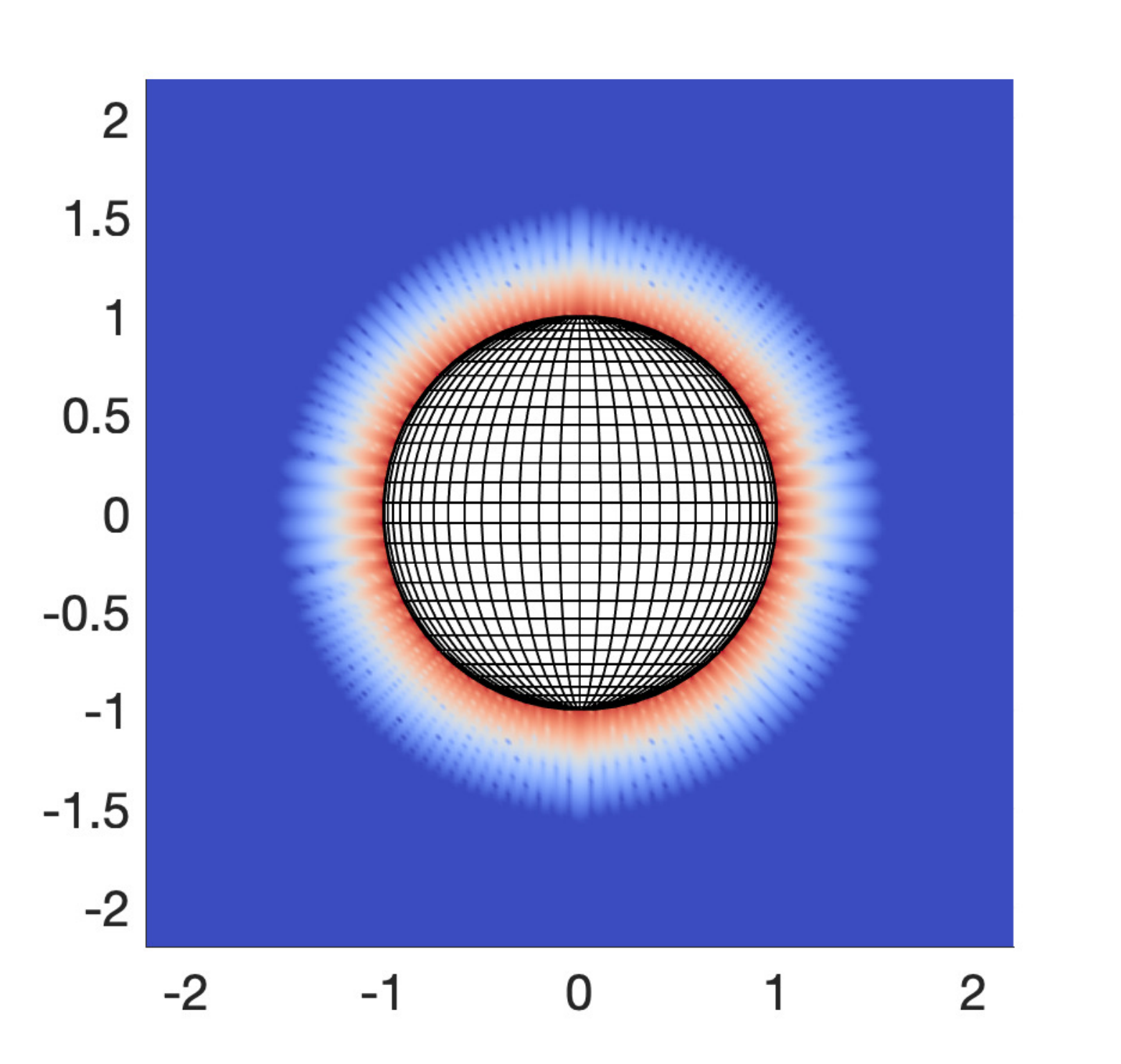}
    \caption{}
    \label{fig:fig1b}
  \end{subfigure}
  \begin{subfigure}{.0615\textwidth}
    \includegraphics[width=1.39\textwidth]{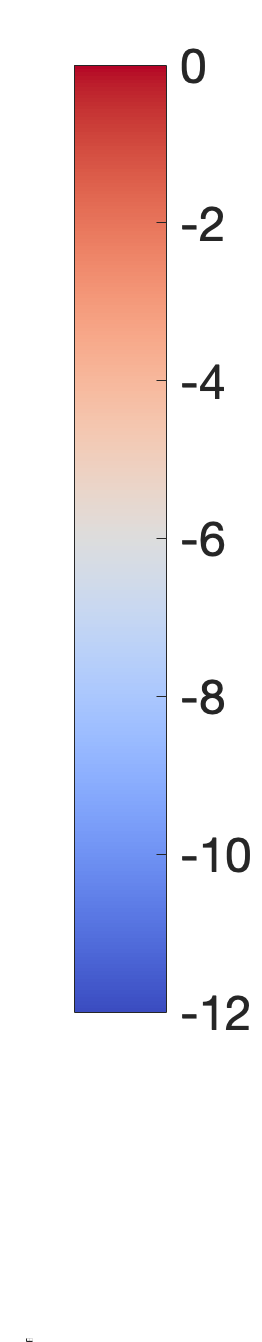}
  \end{subfigure}
 \caption{Error $\EQQ$ ($\log 10$ scale) in computing the harmonic single layer potential with
   unit density on a plane at $y=0$ cutting a sphere of radius 1
   discretized by using (a) the linear mapping, (b) the cosine
   mapping. In both cases, $n_t=30$ and $n_\varphi=60$. }
\label{fig:fig1}
\end{figure}
The linear map clusters the grid points more towards the poles, and at
a fixed distance from the sphere, the error is smaller in these
regions, while the cosine mapping yields a more even error. In both
cases we see oscillations in the error on a length scale of the grid
size.

In Fig. \ref{fig:fig1p2}, the full error estimate $\Eest$ from
\cref{est:3D}, approximated as described in the previous section, is
compared to the measured errors $\EQQ$. The measured errors ($\log 10$
scale) are shown in one selected plane as colored fields, with the
contours of the estimates drawn in black.
\begin{figure}[htbp]
  \centering
  \begin{subfigure}{.43\textwidth}
    \centering
    \includegraphics[width=\textwidth]{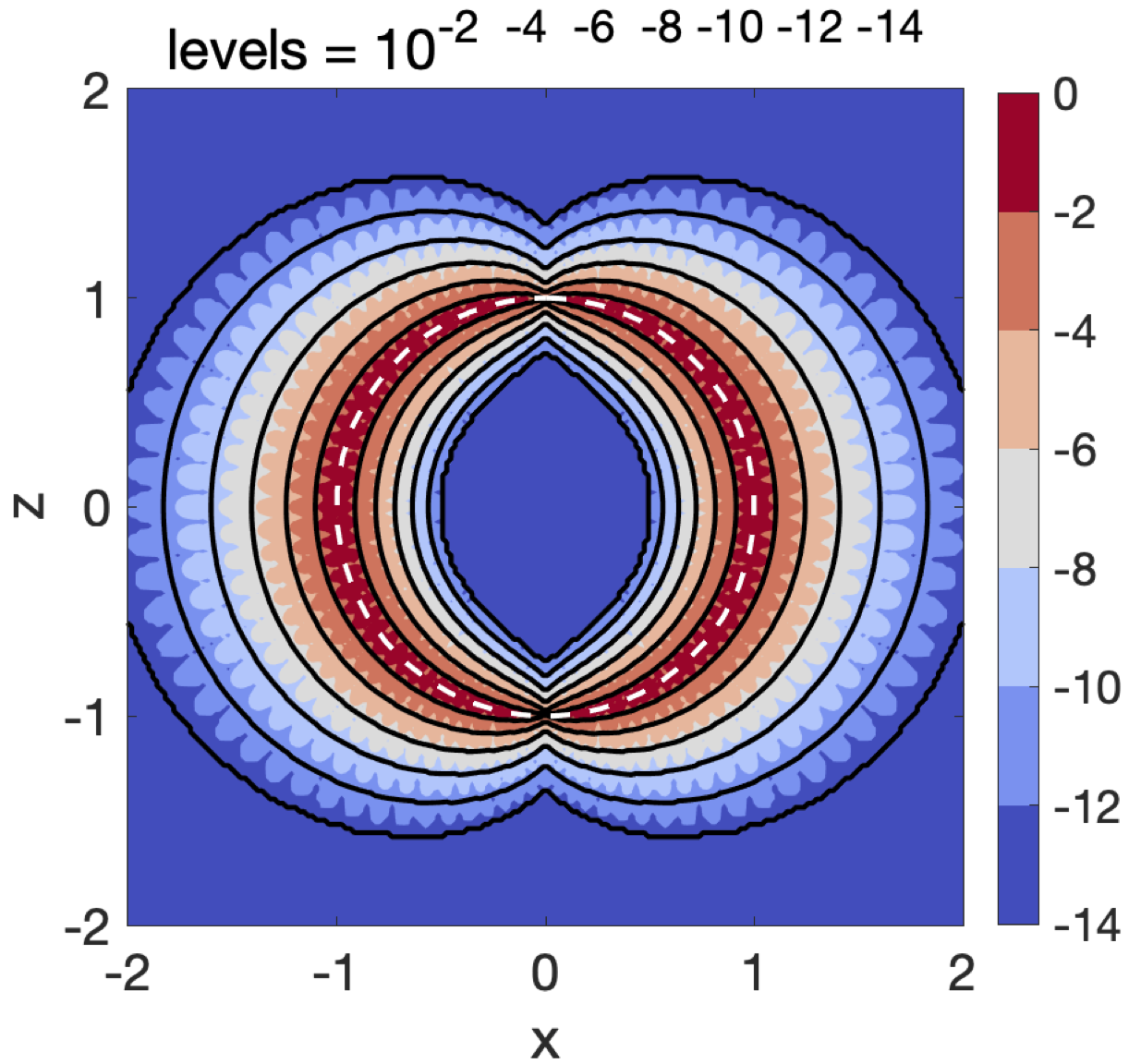}
    \caption{}
   \label{fig:fig_spherr2a}
  \end{subfigure}
   \begin{subfigure}{.43\textwidth}
    \centering
    \includegraphics[width=1\textwidth]{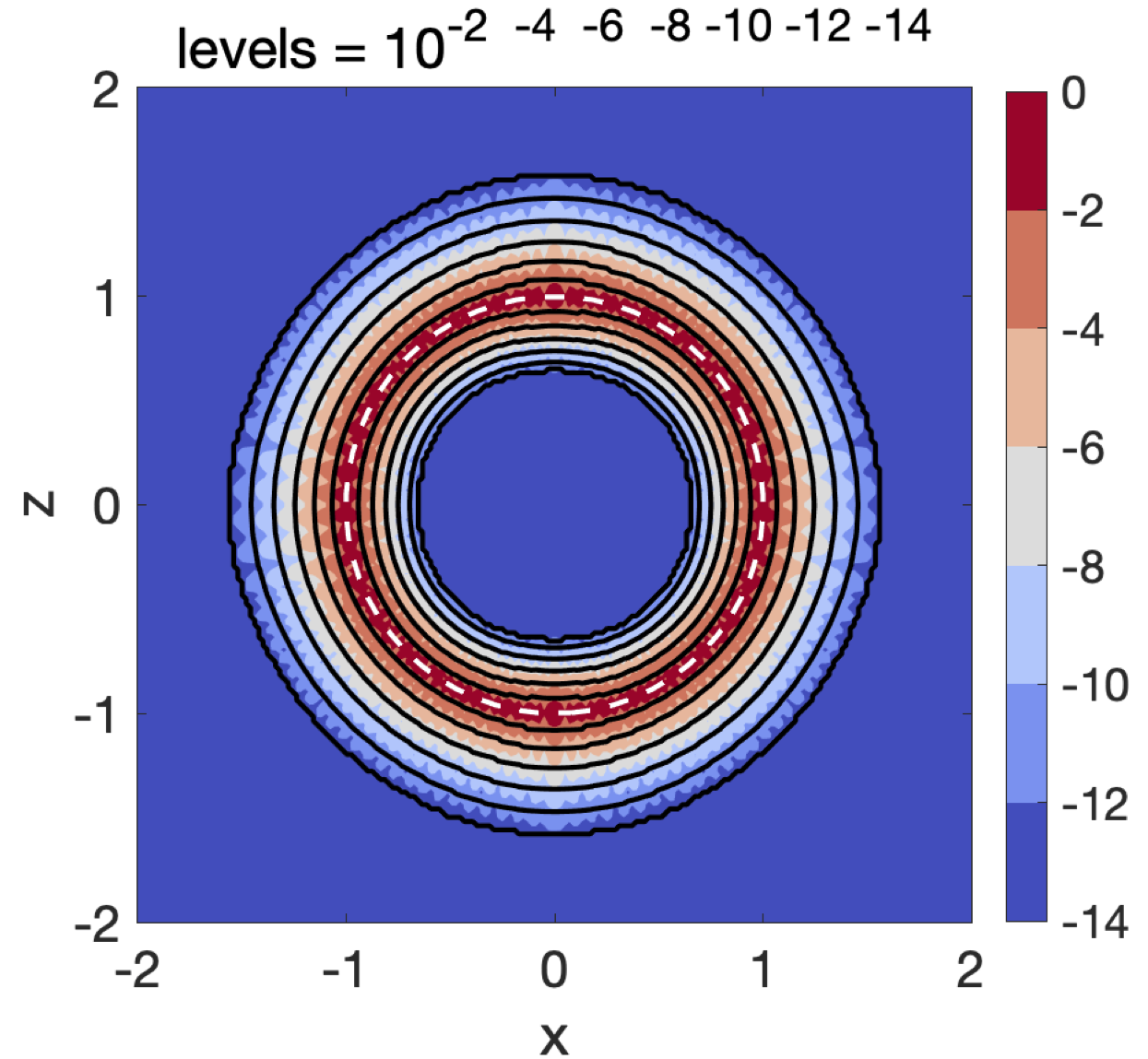}
    \caption{}
    \label{fig:fig_spherr2b}
  \end{subfigure}
\caption{Error $\EQQ$ (colors) and estimates $\Eest$ (black lines) plotted in $\log 10$ scale when computing the
  harmonic single layer potential with unit density over a sphere of radius 1 discretized by using (a) the
  linear mapping, (b) the cosine mapping.
The results are shown for evaluation points in the $xz$-plane for
$y=0$, both inside and outside of the sphere that is indicated with a dashed white line.
}
\label{fig:fig1p2}
\end{figure}
This is done for evaluation points both interior and exterior to the
sphere, and the error estimates can be seen to work well. The
contours of the estimate are smoothly enclosing the oscillatory error,
with an over-estimation that is larger further
out exterior to the sphere. Note however that the last contour is at
an error level of $10^{-14}$, which is a very low level. 

In \cref{est:Esphere}, we derived a simplified error estimate
applicable to this case (when using the cosine map) that depends only
on the radius $a$, the distance from the surface, the number of
discretization points and the half-integer $p$, where $p=1/2$ for the
integral in \eqref{eq:harmonic_sl}.  In Fig. \ref{fig:fig1p3}, we
compare the error measured in a set of evaluation points to the
simplified error estimate, both versus the distance to the surface
(here a negative distance $d$ is used for interior point) for a fixed
grid resolution, and versus the grid resolution for points at a fixed
distance to the sphere.
\begin{figure}[htbp]
  \centering
  \begin{subfigure}{.43\textwidth}
    \centering
    \includegraphics[width=\textwidth]{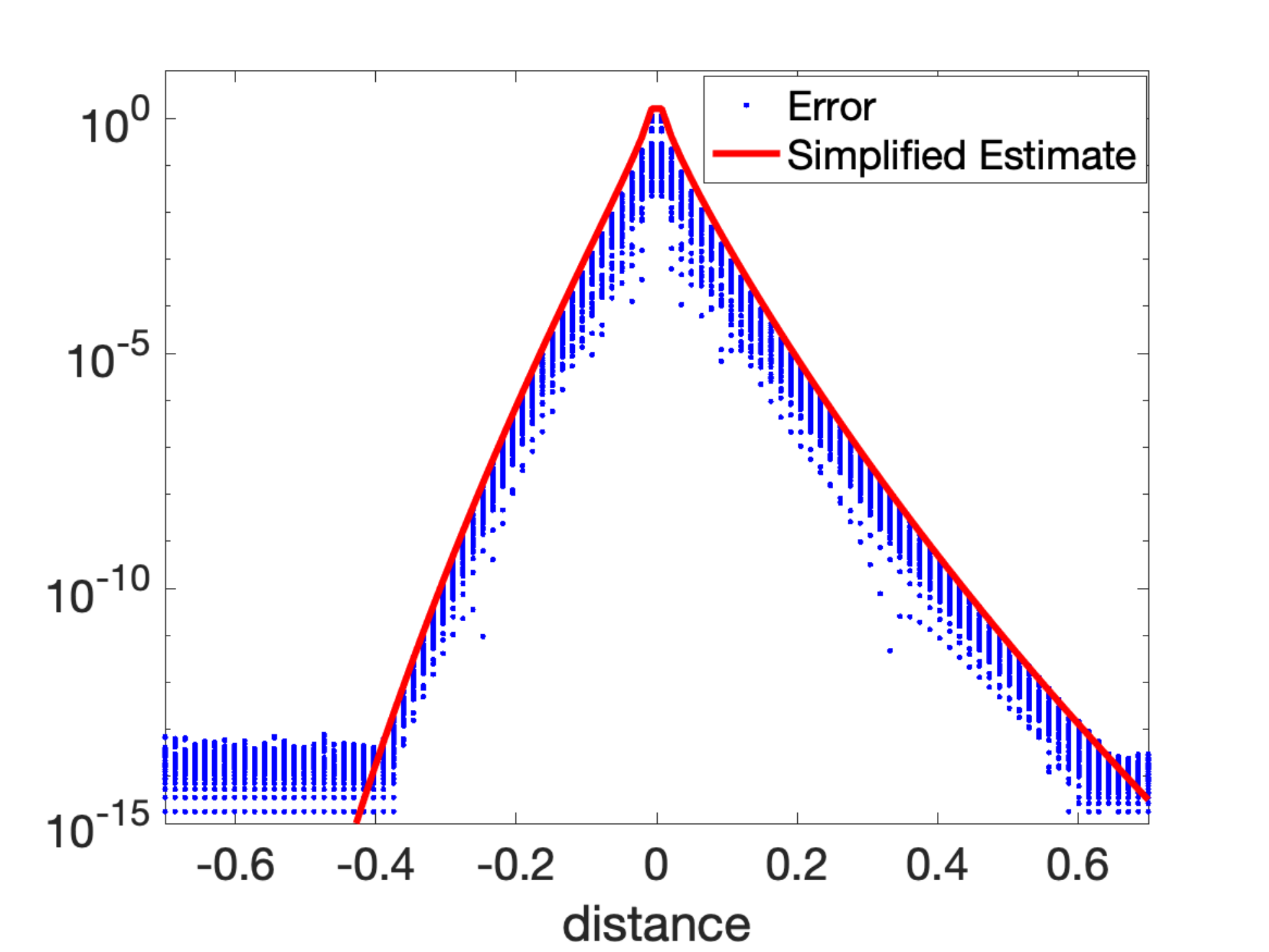}
    \caption{}
    \label{fig:fig_simpspherr2a}
  \end{subfigure}
   \begin{subfigure}{.43\textwidth}
    \centering
    \includegraphics[width=1\textwidth]{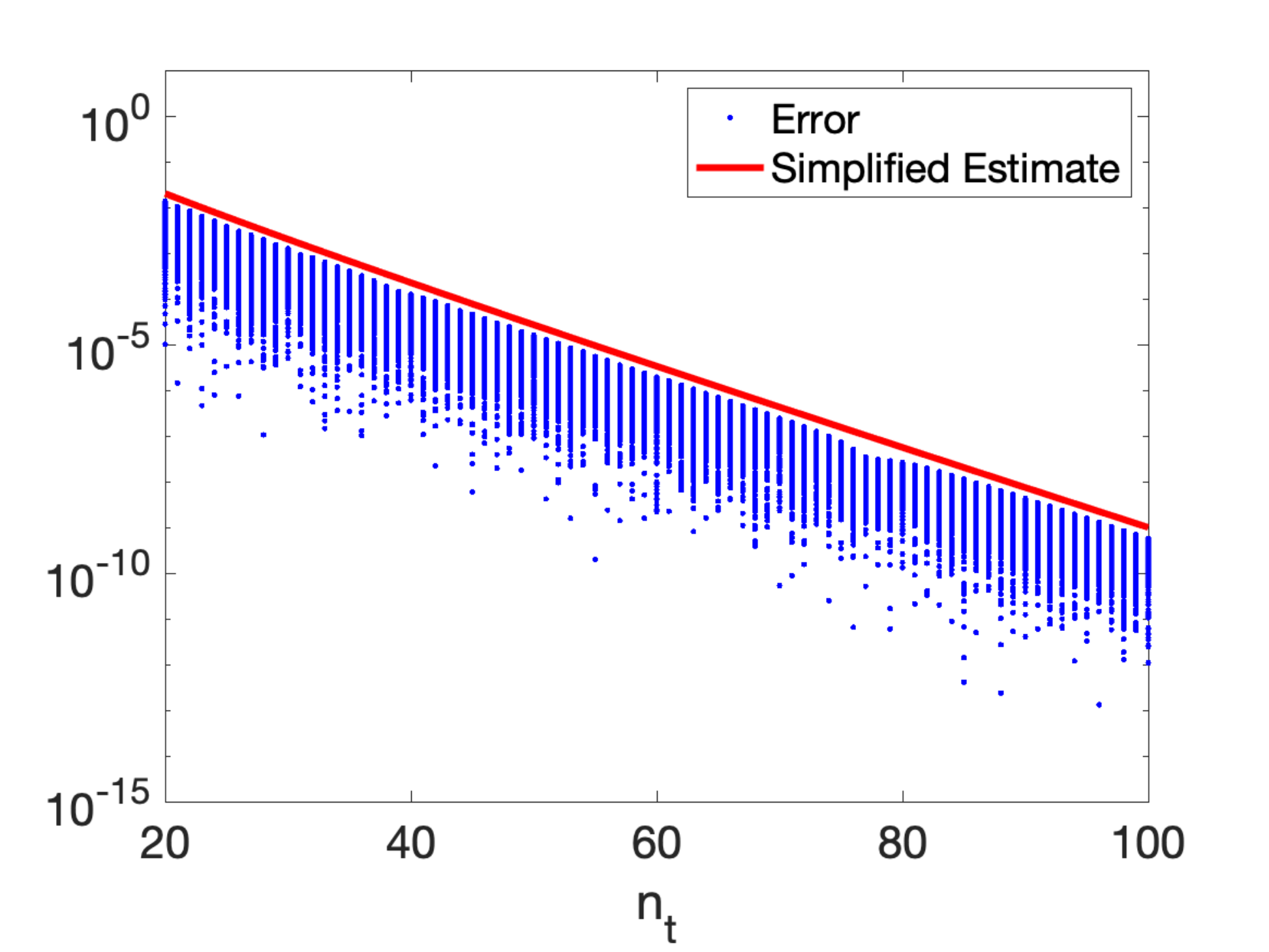}
    \caption{}
    \label{fig:fig_simpspherr2b}
  \end{subfigure}
\caption{Measured quadrature errors $\EQQ$ in evaluating
  \eqref{eq:harmonic_sl} (using the cosine map) at a set of discrete evaluation
  points $\v x \in \reals^3$ (blue dots) and the simplified estimate (red lines);
  (a) varying distance to the sphere (negative values for interior)  and
  fixed $n_t=30$ and (b) varying $n_t$ at the fixed distance
  $d=0.1$. Note that $n=n_\varphi=2n_t$ in the simplified estimate.} 
\label{fig:fig1p3}
\end{figure}
The discrete points are set using the parametrization for a sphere
with radius $(1+d)$, over the full range of the polar angle $0 \le
\theta \le \pi$, but only over an angle sector in the azimuthal angle,
$0 \le \varphi\le \pi/n_\varphi$. The actual errors do not depend on
$\varphi$, more than that there is an oscillation determined by the grid size, and
this range is sufficient to cover the range of errors. Under the
cosine-map, the error is much less dependent on $\theta$ compared to
the linear mapping, but there is a variation, and we include the full
range here. 
From the discrete dots, each representing a different
evaluation point, we can see the range of errors for evaluation points
at the same distance to the sphere. 
The simplified error estimate works better than we could expect, and
gives a rather tight upper bound of the error. 

\subsection{A prolate spheroid}
In the second example we consider an axisymmetric ellipsoid, a prolate
spheroid, with ratio 3-1 between the long and short semi axes. Here the density function is given by 
\begin{equation}
\label{eq:density1}
\sigma(\theta,\varphi)=1+\sin(6\varphi+\theta)\sin^2(\theta).
\end{equation} 
and is in Fig. \ref{fig:ell1a} visualized on the surface by the black
and white colormap. We can see how the varying density breaks the
geometric symmetry of the problem. We first consider the quadrature
error for evaluation points on a vertical wall placed at $y=1.02$,
Fig. \ref{fig:ell1a}-\ref{fig:ell2a},
and then we place random evaluation points around the spheroid
(Fig. \ref{fig:ell1b}), and plot the error vs the estimate in
Fig. \ref{fig:ell2b}. The latter is a simple way to indicate if the
estimate over or under estimate the actual error. The red line
indicates where error and estimate are equal, while the black lines
indicate where they differ by factors 10 and 1/10, respectively.

In Fig. \ref{fig:ell1a}-\ref{fig:ell2a} we are evaluating the harmonic
single layer potential, eq. \eqref{eq:harmonic_sl}. In this case
$p=1/2$ and $f(t,\varphi)= \sigma(t,\varphi) \norm{\dpd{\v\gamma}{t}
  \times \dpd{\v\gamma}{\varphi}}$, where $\sigma(t,\varphi)$ is
obtained by mapping eq. \eqref{eq:density1} with the cosine map.
In Fig. \ref{fig:ell1b}-\ref{fig:ell2b} we consider the harmonic
double layer potential
\begin{equation}
u(\v x) = \int_{\Surf} \frac{{\v n_y} \cdot  (\v y - \v x)\sigma(\v y)
  }{\norm{\v y - \v x}^{3}} \dif S(\v y),
  \label{eq:harmonic_dl}
\end{equation}
for which $p=3/2$ and
$f(t,\varphi)={\v n_\gamma(t,\varphi)} \cdot  (\v \gamma(t,\varphi) - \v x) \sigma(t,\varphi) \norm{\dpd{\v\gamma}{t} \times \dpd{\v\gamma}{\varphi}}$.\\
In the first case $n_t=40$, in the second case $n_t=60$, and for both we set $n_\varphi=2n_t$. In both cases the
estimates can predict very well the actual error. Moreover, it is clear that the density has an effect on the error, but
still the simplification made in
\eqref{eq:ETZ_GFconst}-\eqref{eq:EGL_GFconst} is good
enough to capture the behavior of the overall error.
\begin{figure}[htbp]
  \centering
  \begin{subfigure}{.43\textwidth}
    \centering
    \includegraphics[width=0.9\textwidth]{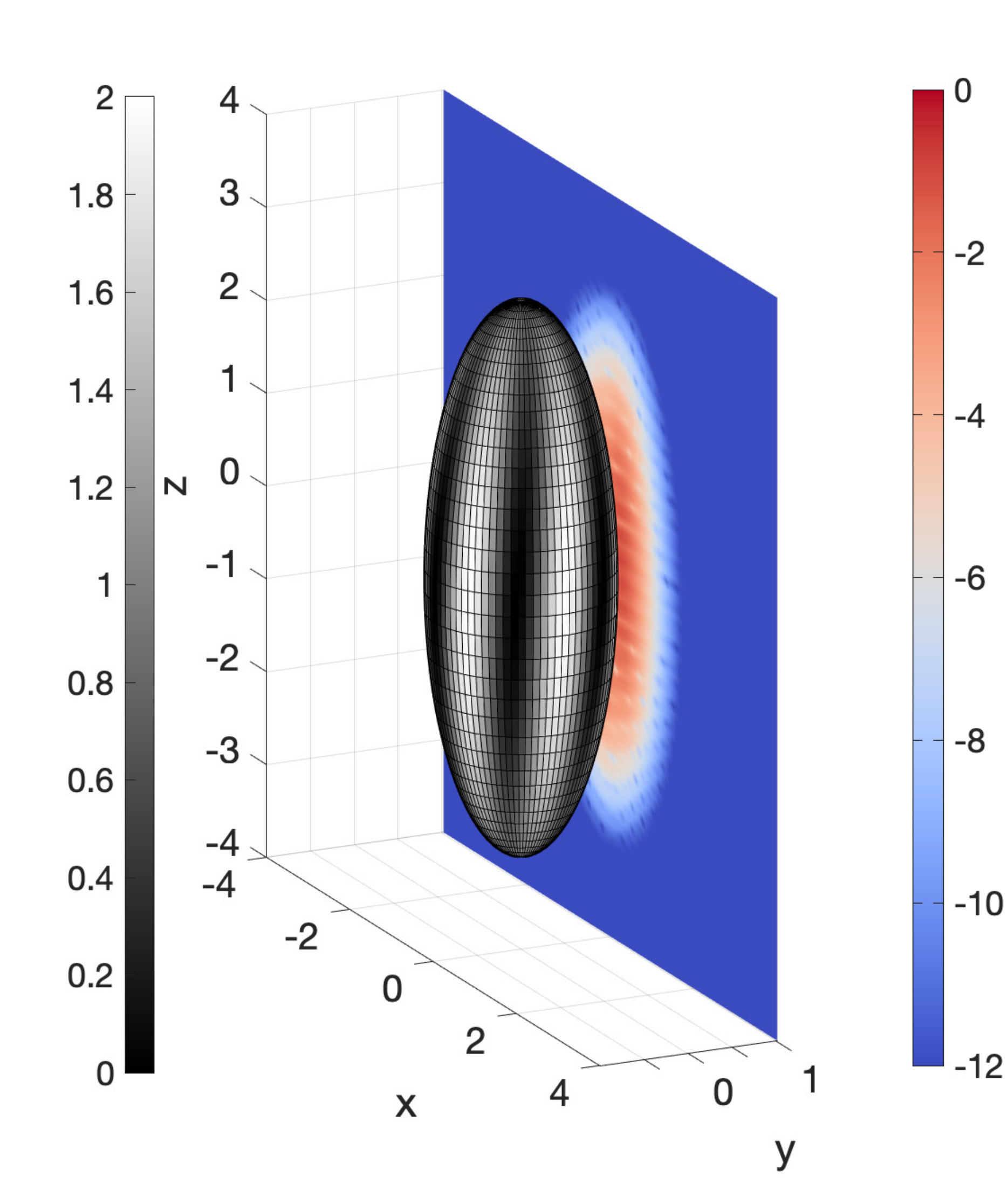}
    \caption{}
    \label{fig:ell1a}
  \end{subfigure}
   \begin{subfigure}{.43\textwidth}
    \centering
    \includegraphics[width=0.9\textwidth]{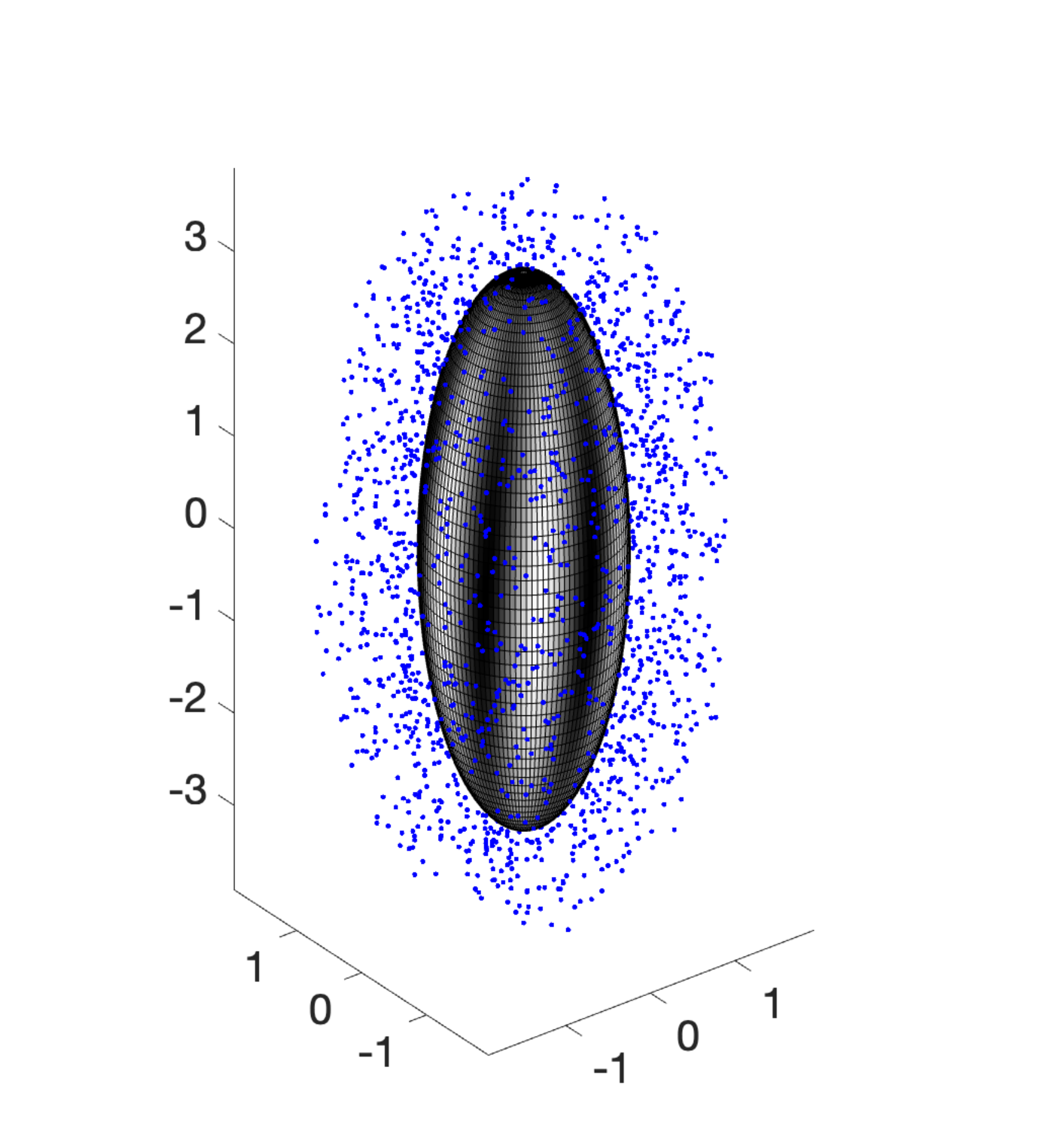}
    \caption{}
    \label{fig:ell1b}
  \end{subfigure}
\caption{Target points considered when computing the harmonic single (a) and the double (b) layer potentials evaluated near a prolate spheroid. The black and white colormap represents the
  density given by eq. \eqref{eq:density1}. The red and blue colormap represents the actual error evaluated on the target wall.}
\label{fig:fig1p5}
\end{figure}

\begin{figure}[htbp]
  \centering
  \begin{subfigure}{.45\textwidth}
    \centering
   \includegraphics[width=1\textwidth]{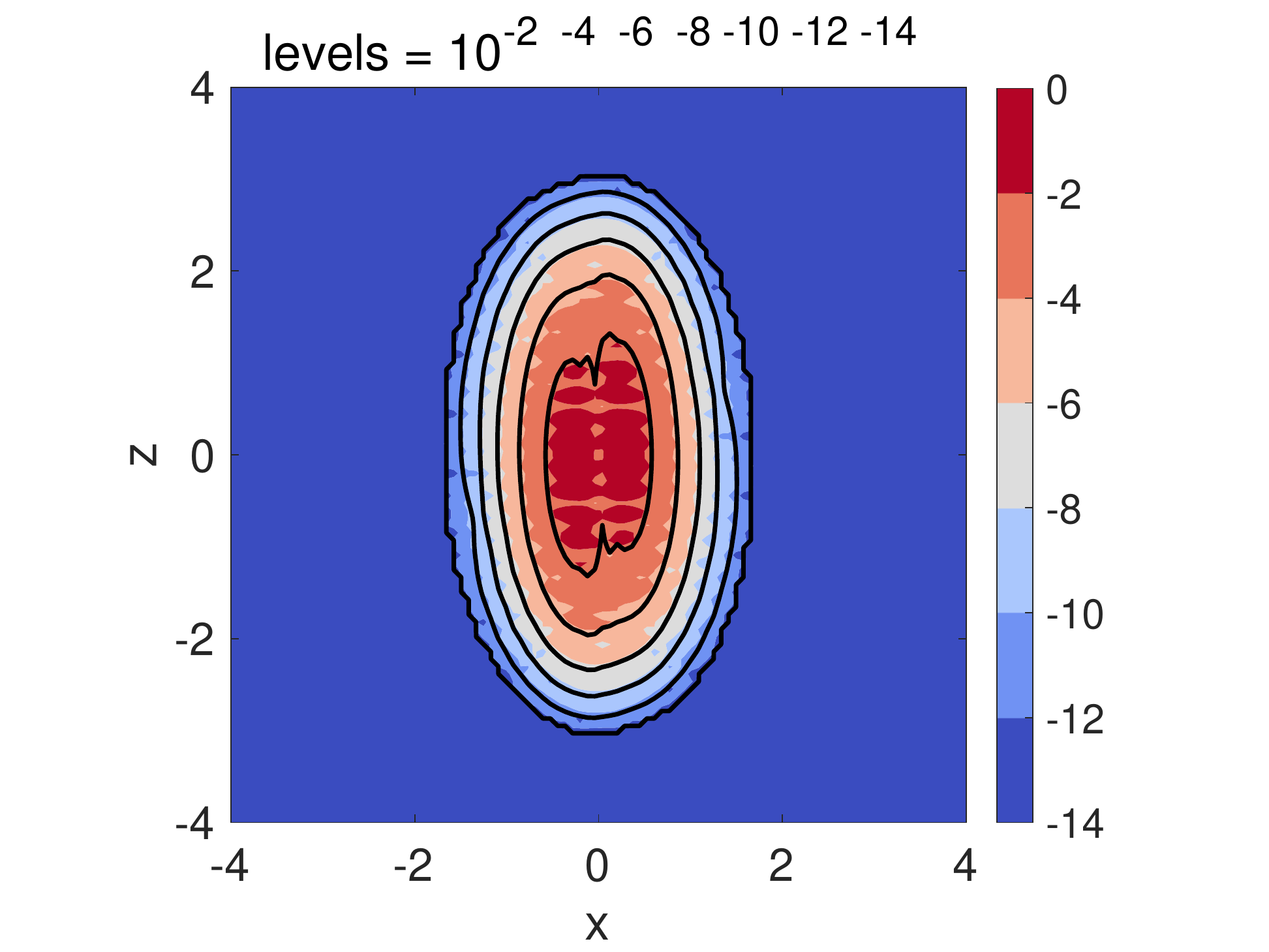}
    \caption{}
    \label{fig:ell2a}
  \end{subfigure}
   \begin{subfigure}{.45\textwidth}
   \centering
        \includegraphics[width=1\textwidth]{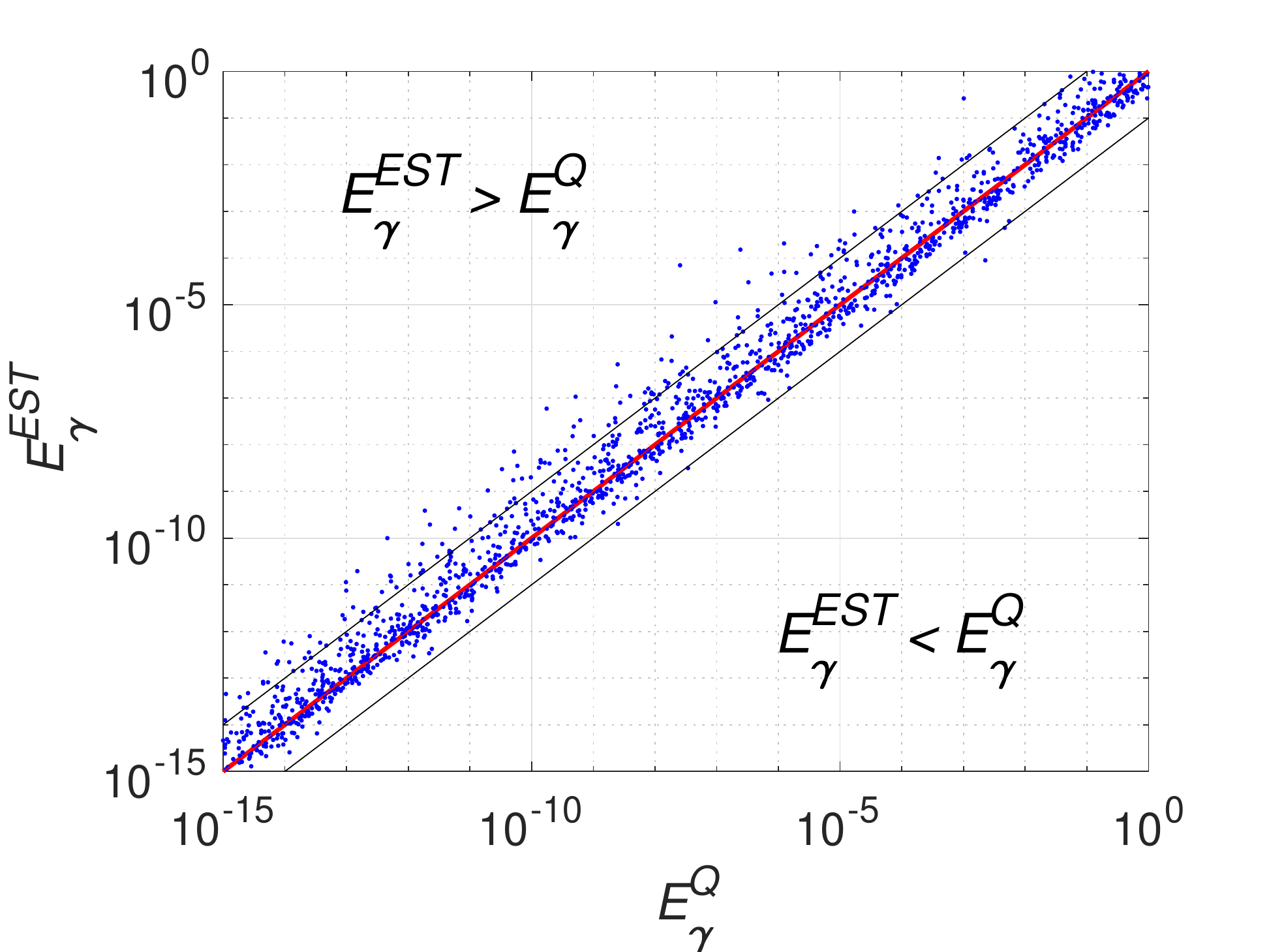}
    \caption{}
    \label{fig:ell2b}
  \end{subfigure}
\caption{(a) Error $\EQQ$ (colors) and estimates $\Eest$ (black lines) plotted in $\log 10$ scale when computing
  the harmonic single layer potential near the ellipsoid on a plane at
  $y=1.02$, shown in Fig. \ref{fig:ell1a}. (b) Estimates vs error in computing the harmonic double layer
  potential at random evaluation points showed in Fig. \ref{fig:ell1b}.
The three lines from top to bottom indicate where the estimate of
$\Eest$ is a factor of $10$, $1$ and $1/10$ times the measured value
of $\EQQ$. If the estimate was never underestimating the error, no
dot would fall below the red line.
}
\label{fig:fig1p6}
\end{figure}

\subsection{Non-axisymmetric geometry}
In the last example we consider a non-axisymmetric geometry given by
\begin{equation}
\v\gamma^\circ(\theta,\varphi)=
\begin{cases}
\rho(\theta,\varphi)\cos(\varphi)\sin(\theta)\\
\rho(\theta,\varphi)\sin(\varphi)\sin(\theta)\\
\rho(\theta,\varphi)\cos(\theta)
\end{cases}
\label{eq:nonaxisurface}
\end{equation}
with
$\rho(\theta,\varphi)=0.8+0.2e^{-3\text{Re}(Y_3^2(\theta,\varphi))}$ and $Y_{3}^{2}(\theta ,\varphi )={1 \over 4}{\sqrt {105 \over 2\pi }} e^{2i\varphi } \sin ^{2}(\theta) \cos(\theta) $.
The surface is enclosed in a spherical shell of radius $a=1.46$, as shown in Fig. \ref{fig:nonaxi_shell}. Here we consider the modified Helmholtz equation $(\Delta-\omega^2)u=0$, and compute the corresponding single layer potential:
 \begin{equation}
 u(\v x) = \int_{\Surf} \frac{e^{-\omega \norm{\v y - \v x}} \sigma(\v y)
  }{\norm{\v y - \v x}} \dif S(\v y).
\end{equation}
We use the cosine mapping and define $\v \gamma(t,\varphi)=\v\gamma^\circ(\cos^{-1}(-t),\varphi)$.
For this case, referring to eq. \eqref{eq:ETZplusEGL}, $p=1/2$ and $f(t,\varphi)= e^{-\omega \norm{\v \gamma(t,\varphi) - \v x}}\sigma(t,\varphi) \norm{\dpd{\v\gamma}{t} \times \dpd{\v\gamma}{\varphi}}$.
\begin{figure}[htbp]
\centering
    \includegraphics[width=0.5\textwidth]{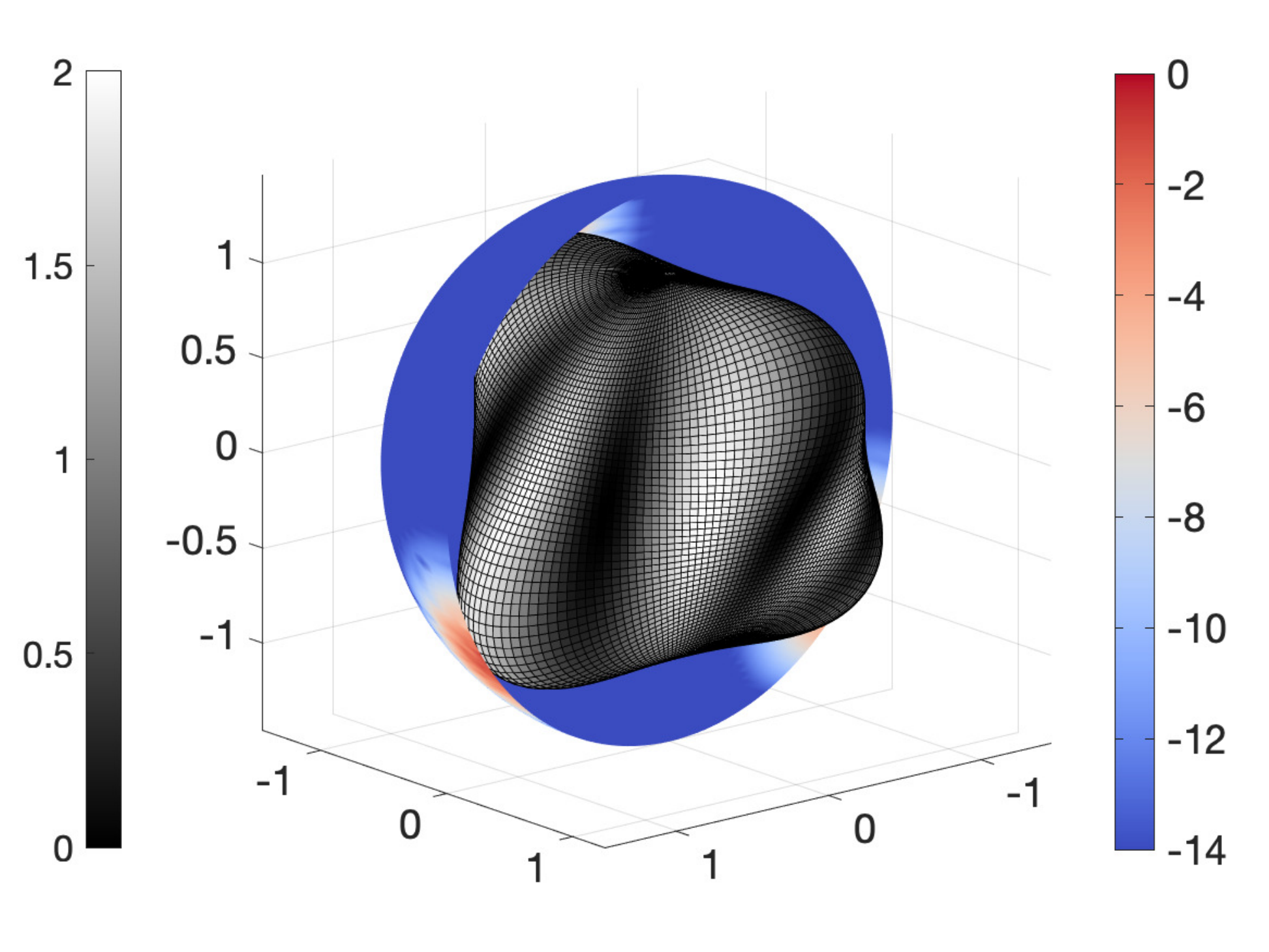}
    \caption{Half of the spherical shell enclosing the non
      axisymmetric shape defined in \eqref{eq:nonaxisurface}. The
      black and white colormap represents the density given by
      eq. \eqref{eq:density1}. }
    \label{fig:nonaxi_shell}
\end{figure}
We consider the case $\omega=3$ and the density function $\sigma$
given by eq. \eqref{eq:density1}. In Fig. \ref{fig:nonaxia} we show
the error and the estimates computed on the whole spherical shell,
plotted here with the horizontal axis being the the azimuthal angle
and the vertical axis being the polar angle. In Fig. \ref{fig:nonaxib}
we zoom in on the white rectangle highlighted in Fig. \ref{fig:nonaxia},
to better show the agreement between estimate and error.
\begin{figure}[htbp]
  \centering
  \begin{subfigure}{.6\textwidth}
    \centering
    \includegraphics[height=3.5cm]{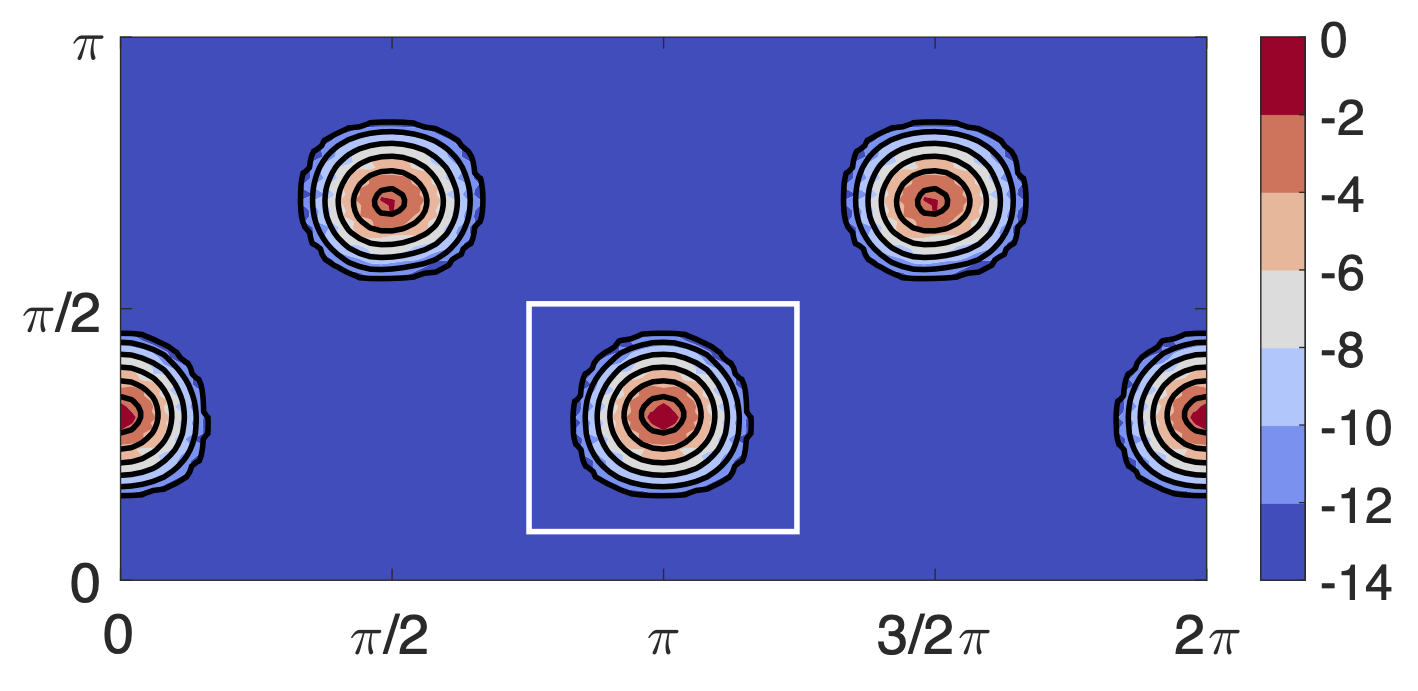}
    \caption{}
    \label{fig:nonaxia}
  \end{subfigure}
   \begin{subfigure}{.38\textwidth}
    \centering
    \includegraphics[height=3.5cm]{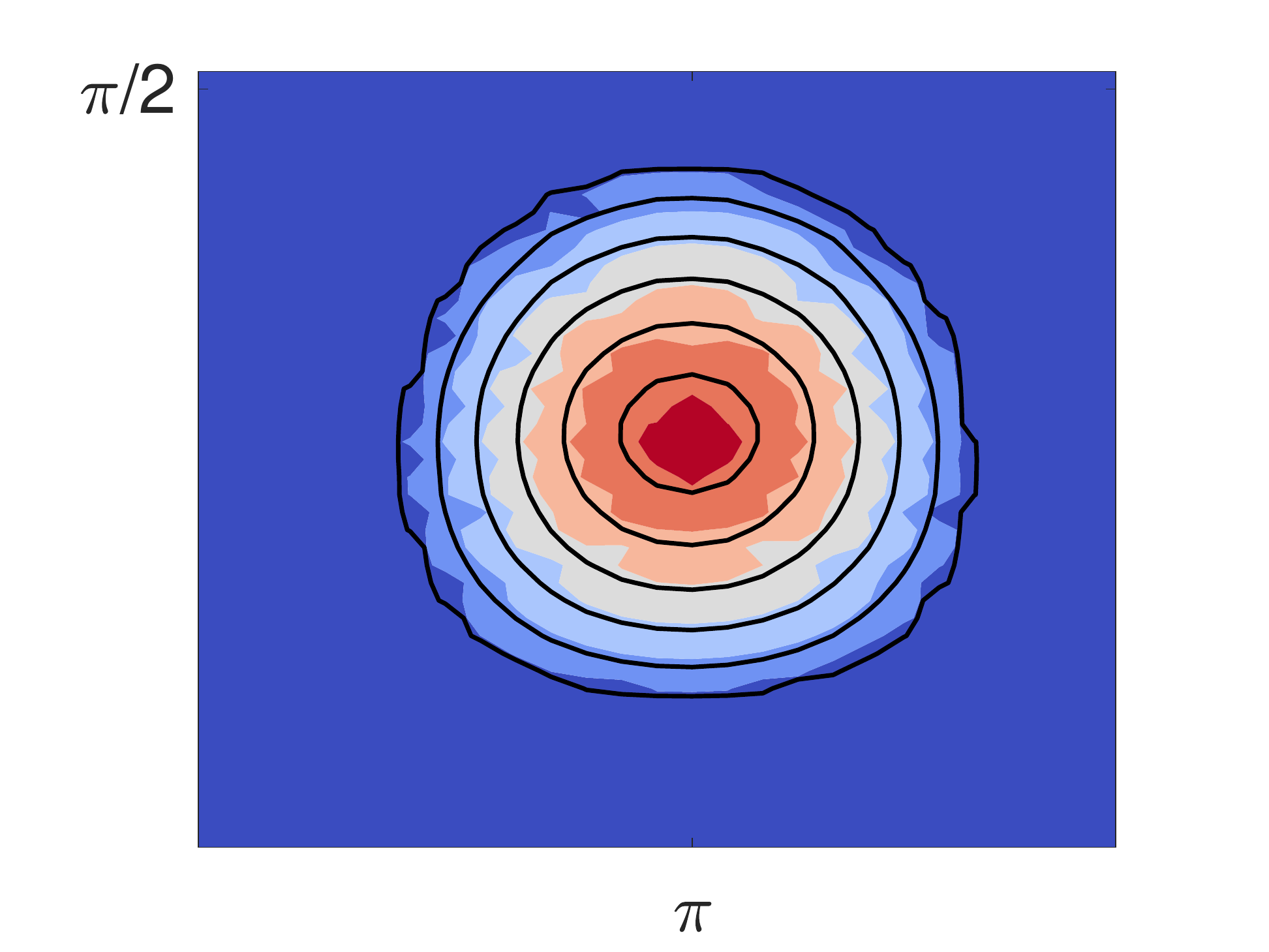}
    \caption{}
    \label{fig:nonaxib}
  \end{subfigure}
\caption{Error $\EQQ$ (colors) and estimate $\Eest$ (black lines) in computing the
  modified Helmholtz potential with $\omega=3$ on the shell enclosing
  the non axisymmetric geometry eq. \eqref{eq:nonaxisurface}. The
  levels for the estimates contours are $10^{-2}, 10^{-4}, 10^{-6},
  10^{-8}, 10^{-10}, 10^{-12}, 10^{-14}$. In (b) we zoom in on the white
  rectangle drawn in (a).
}
\label{fig:fig1p6vfvf}
\end{figure}

\section{Conclusions}
\label{sec:conclusions}

In this paper, we studied the error incurred by numerically
approximating layer potentials over surfaces of spherical topology.
We have derived error estimates for discretizations with the
trapezoidal rule in the azimuthal angle, and a Gauss-Legendre rule in a
variable that maps to the polar angle.  The framework for
the derivation of the error estimates, and the practical evaluation
there of, were introduced in \cite{AFKLINTEBERG20221} for surfaces of
genus 1, discretized by either a global trapezoidal rule in both
directions, or a panel based Gauss-Legendre rule.
Here we extended this approach with special attention to the global
parametrization. There is one component of the error estimate that
cannot be directly evaluated for evaluation points on the symmetry
axis of an axisymmetric surface. Starting by deriving analytical expressions for
the roots of a squared distance function for a sphere, we were able to
prove that this contribution vanishes at these points. We could also
derive a simplified error estimate for the sphere, that shows the
decay in error with the distance of the evaluation point to the sphere
with a simple formula.  Some analytical results were also extended to
the more general case of an axisymmetric surface, and we devised a
strategy for evaluating the error estimate for general surfaces,
avoiding the difficulties associated with the discretization around
the poles. 

The error estimate does not have any unknown coefficients, but for
each evaluation point for a general surface, two complex roots to the squared distance
function must be computed using one-dimensional root finding.
In numerical experiments, we have shown that the error estimate indeed
estimates the actual error quite well, also for moderate numbers of
discretization points. This is true for different layer potentials,
various surfaces, and with a variable layer density. The simplified
error estimate for the sphere is shown to give a tight upper bound for
the error at a given distance from the spherical surface.

\begin{appendices}
\section{Derivations for the Gauss-Legendre error} 
\label{app:GLders}
 We consider a sphere of radius $a$ and the associated
  Gauss-Legendre error as defined in \cref{eq:EGLdefwEfac} with
  $E_{fac}^{GL}(\v x,\varphi)$ defined in \cref{eq:EfacGL}.
Under the map $t=-\cos \theta$, the squared distance function 
evaluates as
\[R^2(t,\varphi,\v x)=a^2-2a(\sqrt{1-t^2}(x\cos \varphi +
  y \sin \varphi)+tz)+x^2+y^2+z^2.
\]
and $\geomfac{1}\pars{t,\varphi ,\v x}=(\partial R^2/\partial t)^{-1}$.
The root $\theta_0(\varphi,\v x)$ is given in
\cref{eqn:thetaroot_given_phi2}, and
$t_0(\varphi,\v x)=-\cos(\theta_0(\varphi,\v x))$.

We start by considering an evaluation point at the symmetry axis,
i.e. $\v x=(0,0,z)$, $z \ne a$. 
For this case we get $\theta_0=\pm i \ln(\abs{z}/a)$ as given in 
\cref{corr:thetaroot_given_phi_on_zaxis}.
We get $t_0=(\delta+1/\delta)/2$, where we let $\delta=a/\abs{z}$ if
$\abs{z}>a$, and $\delta=\abs{z}/a$ if $\abs{z}<a$, such that $\delta>1$.
Hence, $\abs{\sqrt{t_0^2-1}}=(\delta-1/\delta)/2$, and 
$\abs{t_0+\sqrt{t_0^2-1}}=\delta$.
Finally, we have $\geomfac{1}\pars{t,\varphi ,\v x}=1/(2a\abs{z})$, and combined
this yields the expression  for $E_{fac}^{GL}(\v x,\varphi)$ given in 
  \cref{eqn:EfacGLzaxis}.

  Now, we instead consider an evaluation point at the equator,
  $\v x =(0,y,0)$. We could however equally well pick $\v x= (x,0,0)$,
  or $\v x= (x,y,0)$ and would obtain the same final result
  with $\| \v x\|=\rho$.  With $\v x =(0,y,0)$ we get
\begin{align}
  E_{fac}^{GL}(\v x,\varphi) =\frac{1}{(2a\abs{y}\abs{\sin \varphi})^p}
  \frac{\abs{\sqrt{t_0(\varphi)^2-1}}}{\abs{t_0(\varphi)}^p}
\abs{ \frac{1}{t_0(\varphi) + \sqrt{t_0(\varphi)^2-1}}}^{2n+1}.
\label{eq:EfacGLsimpl}
\end{align}

The root $t_0(\varphi)=-\cos(\theta_0(\varphi))$ where
$\theta_0(\varphi)$ is defined
in \cref{lemma:thetaroot_given_phi}, in \cref{eqn:thetaroot_given_phi}.
With  $\v x =(0,y,0)$, $\lambda$ in that expression simplifies to
$\lambda=(\abs{y}/a+a/\abs{y})/(2\abs{\sin \varphi}=(\delta+1/\delta)/ (2 \abs{\sin \varphi})$
  where we let $\delta=\abs{y}/a$ if $\abs{y}>a$ and $\delta=a/\abs{y}$ if
  $\abs{y}<a$, such that $\delta>1$. 
The expression for $\lambda^2-1$ then becomes the same as in 
\cref{eqn:lambdasq_m1_tz}, but with $\varphi$ instead of $\theta$. The
peak of the error is at the closest point to $\v x=(0,y,0)$, i.e. at
$\varphi=\pi/2$, and also here, we ignore the last term in the 
expression for $\lambda^2-1$. With this we get that
$\lambda+\sqrt{\lambda^2-1} \approx \delta/\abs{\sin
\varphi}$. Introducing $\tilde{\delta}=\delta/\abs{\sin
\varphi}$, and noting that the square roots are evaluated at points
away from the branch cut, we have 
\[
  t_0=i(\tilde{\delta}-1/\tilde{\delta})/2
  \qquad \sqrt{t_0^2-1}=i(\tilde{\delta}+1/\tilde{\delta})/2
  \qquad  t_0+\sqrt{t_0^2-1}=i\tilde{\delta}.
  \]
Similarly to the derivation based on the trapezoidal error, we evaluate all
terms in  \cref{eq:EfacGLsimpl} except the last term at
$\varphi=\pi/2$. We then have

\begin{align}
  E_{fac}^{GL}(\v x,\varphi) \approx \frac{1}{2} \frac{1}{(a\abs{y})^p}
  \frac{\abs{\delta+1/\delta}}{\abs{\delta-1/\delta}^p}
\frac{\abs{\sin \varphi}^{2n+1}}{\delta^{2n+1}}.
\label{eq:EfacGLsimpl2}
\end{align}
Inserting into \cref{eq:EGLdefwEfac} and using
\[
\int_0^{2\pi} \abs{\sin \varphi}^{2n_t+1} \dif \varphi =\int_0^{2\pi}
\abs{\cos \varphi}^{2n_t+1} \dif \varphi =2 \int_{-\pi/2}^{\pi/2} (\cos \varphi)^{2n_t+1} \dif \varphi 
  \]
we can identify the integral in  \cref{eqn:cosbetaint}. With $f=a^2$ the total
result becomes $E^{GL}(\v\gamma,a^2,p,2n_t,\v x)$ as given in  \cref{eqn:sph_EGL_approx}.
\end{appendices}

\section*{Acknowledgments}
A.-K.T. acknowledges the support by the Swedish Research Council
 under grant no 2019-05206.

\bibliography{library,library-local}


\end{document}